\documentclass[12pt,svgnames]{article}
\usepackage{amsmath}
\usepackage{amssymb}
\usepackage{amsthm}
\usepackage{wasysym}
\usepackage{enumerate}
\usepackage{dsfont}
\usepackage{stmaryrd}
\usepackage{comment}
\usepackage{tikz}
\usepackage{quiver}
\usepackage{tikz-cd}
\usepackage{extarrows}
\usepackage{mathrsfs}
\usepackage{enumitem}
\usepackage{fullpage}
\usepackage{hyperref}
\usepackage{titlesec}
\usepackage[all]{xy}
\hypersetup{
    linkbordercolor = {PaleTurquoise},
    citebordercolor = {PaleGreen},
}

\titleformat{\section}{\normalsize\bfseries}{\thesection}{1em}{}
\titleformat{\subsection}{\normalsize\bfseries}{\thesubsection}{1em}{}

\numberwithin{equation}{subsection}

\newtheorem{defn}{Definition}[section]
\newtheorem{para}[defn]{}
\newtheorem{prop}[defn]{Proposition}
\newtheorem{theo}[defn]{Theorem}

\newtheorem{rmk}[defn]{Remark}

\newtheorem{egg}[defn]{Example}

\newtheorem*{claim*}{Claim}
\newtheorem*{just*}{Justification}
\newtheorem*{assumption*}{Assumption}
\newtheorem*{lem*}{Lemma}
\newtheorem*{prop*}{Proposition}
\newtheorem*{thm*}{Theorem}

\newcommand{\V}{\mathscr{V}}

\newcommand{\ob}{\mathop{\mathsf{ob}}}

\newcommand{\tensor}{\otimes}

\newcommand{\Cat}{\text{-}\mathsf{Cat}}

\newcommand{\underJ}{\kern -0.5ex \mathscr{J}}
\newcommand{\Set}{\mathsf{Set}}

\newcommand{\T}{\mathbb{T}}
\newcommand{\C}{\mathscr{C}}

\newcommand{\A}{\mathscr{A}}

\newcommand{\Mod}{\text{-}\mathsf{Mod}}

\newcommand{\scrS}{\mathscr{S}}

\newcommand{\Xbar}{\overline{X}}

\newcommand{\E}{\mathsf{E}}

\newcommand{\Var}{\mathsf{Var}}
\newcommand{\Pos}{\mathsf{Pos}}

\newcommand{\Str}{\mathsf{Str}}

\newcommand{\bbS}{\mathbb{S}}

\newcommand{\calS}{\mathcal{S}}

\newcommand{\PMet}{\mathsf{PMet}}
\newcommand{\Met}{\mathsf{Met}}

\newcommand{\sfk}{\mathsf{k}}

\newcommand{\Convex}{\mathsf{Convex}}
\newcommand{\Preord}{\mathsf{Preord}}
\newcommand{\Gph}{\text{-}\mathsf{Gph}}
\newcommand{\RGph}{\text{-}\mathsf{RGph}}
\newcommand{\Rbar}{\overline{R}}
\newcommand{\Sbar}{\overline{S}}

\begin{document}

\title{\Large \textbf{Exponentiability in categories of relational structures}}
\author{Jason Parker
\medskip \\
\small Brandon University, Brandon, Manitoba, Canada}
\date{}

\maketitle

\begin{abstract}
For a relational Horn theory $\T$, we provide useful sufficient conditions for the exponentiability of objects and morphisms in the category $\T\Mod$ of $\T$-models; well-known examples of such categories, which have found recent applications in the study of programming language semantics, include the categories of preordered sets and (extended) metric spaces. As a consequence, we obtain useful sufficient conditions for $\T\Mod$ to be cartesian closed, locally cartesian closed, and even a quasitopos; in particular, we provide two different explanations for the cartesian closure of the categories of preordered and partially ordered sets. Our results recover (the sufficiency of) certain conditions that have been shown by Niefield and Clementino--Hofmann to characterize exponentiability in the category of partially ordered sets and the category $\mathscr{V}\text{-}\mathsf{Cat}$ of small $\V$-categories for certain commutative unital quantales $\V$. 
\end{abstract}

\section{Introduction}

An object $X$ of a category $\C$ with finite products is \emph{exponentiable} if the product functor $X \times (-) : \C \to \C$ has a right adjoint, while a morphism $f : X \to Y$ of a category $\C$ with finite limits is \emph{exponentiable} if the object $f : X \to Y$ of the slice category $\C/Y$ is exponentiable, or equivalently if the pullback functor $f^* : \C/Y \to \C/X$ has a right adjoint (see \cite[Corollary 1.2]{Cartesianness}). A category with finite products is \emph{cartesian closed} if every object is exponentiable, and a category with finite limits is \emph{locally cartesian closed} if every morphism is exponentiable. The exponentiable objects and morphisms of many different categories have been studied and characterized in the literature. For example, the exponentiable objects in the category $\mathsf{Top}$ of topological spaces and continuous maps were characterized by Day and Kelly in \cite{Day_Kelly_quotients}. In \cite{Cartesianness}, Niefield characterized the exponentiable morphisms of $\mathsf{Top}$ and also of the categories $\mathsf{Unif}$ and $\mathsf{Aff}$ of respectively \emph{uniform spaces} and \emph{affine schemes}, while in \cite{Exp_morphisms} Niefield characterized the exponentiable morphisms of the category $\mathsf{Pos}$ of partially ordered sets and monotone maps. In \cite{Exponentiation_in_V_categories, Rise_and_fall}, Clementino and Hofmann characterized the exponentiable objects and morphisms of the category $\V\Cat$ of (small) $\V$-categories and $\V$-functors for a \emph{commutative unital quantale} $\V$ whose underlying complete lattice is a complete Heyting algebra. In particular, they characterized the exponentiable objects and morphisms of many categories of (pseudo-)metric spaces.

The present paper provides a further contribution to the study of exponentiability, in the context of categories of \emph{relational structures}, which have recently attracted interest in the study of programming language semantics (see \cite{Monadsrelational, Quant_alg_reasoning}). Specifically, we provide useful sufficient conditions for the exponentiability of objects and morphisms in the category $\T\Mod$ of $\T$-models for a \emph{relational Horn theory} $\T$. As we recall in Example \ref{without_equality_examples}, prominent examples of such categories include: the categories $\mathsf{Preord}$ and $\mathsf{Pos}$ of preordered and partially ordered sets (with monotone maps); for a commutative unital quantale $\V$, the categories $\V\Gph$, $\V\RGph$, $\V\Cat$, $\PMet_\V$, and $\Met_\V$ of respectively (small) \emph{$\V$-graphs}, \emph{reflexive $\V$-graphs}, \emph{$\V$-categories}, \emph{pseudo-$\V$-metric spaces}, and \emph{$\V$-metric spaces}. Our results recover (the sufficiency of) the conditions for exponentiability established for $\mathsf{Pos}$ (by Niefield) and for $\V\Cat$ (by Clementino--Hofmann). We also provide useful sufficient conditions for categories of relational structures to be cartesian closed, locally cartesian closed, and even quasitoposes. In particular, we offer two different explanations for why the categories $\Preord$ and $\Pos$ of preordered and partially ordered sets are cartesian closed.  

We now outline the paper. After recalling some categorical background in \S\ref{background} about exponentiability (in particular, a useful characterization of exponentiability of morphisms established by Dyckhoff and Tholen in \cite{Dyckhoff}), in \S\ref{first_section} we recall some relevant background on relational Horn theories and categories of relational structures from \cite{Extensivity} (and \cite{Monadsrelational}). In \S\ref{exp_Str} we first show that if $\Pi$ is any \emph{relational signature}, then the category $\Str(\Pi)$ of \emph{$\Pi$-structures} and \emph{$\Pi$-morphisms} is always locally cartesian closed, and even a quasitopos (Theorem \ref{Str_lcc}). 

In \S\ref{exponentiable_Tmod} we begin to study exponentiability in $\T\Mod$ for a relational Horn theory $\T$ over a relational signature $\Pi$; unlike the situation for $\Str(\Pi)$, in general $\T\Mod$ is not even cartesian closed, let alone locally cartesian closed or a quasitopos (however, $\T\Mod$ is always \emph{symmetric monoidal closed}; see Remark \ref{closed_structure}). In many examples of relational Horn theories, the relational signature $\Pi$ carries a (pointwise) preorder that manifests in the axioms of the theory, and so we consider \emph{preordered relational signatures} (Definition \ref{preord_sig}). We say that a preordered relational signature $\Pi$ is \emph{discrete} or a \emph{complete Heyting algebra} if the (pointwise) preorder on $\Pi$ is respectively discrete or a complete Heyting algebra. 

In \S\ref{disc_section} we identify a useful sufficient condition for objects and morphisms of $\T\Mod$ to be exponentiable, which we call \emph{convexity} (Definition \ref{convex_unified}), when $\T$ is a \emph{reflexive} relational Horn theory over a discrete relational signature $\Pi$. Central examples of such relational Horn theories include the theories for preordered and partially ordered sets; in particular, a morphism of preordered or partially ordered sets is convex iff it is an \emph{interpolation-lifting map} in the sense of Niefield \cite[Definition 2.1]{Exp_morphisms} (see Example \ref{convex_egg}). When the axioms of $\T$ have certain properties (see Definition \ref{safe_axiom}), we can show that all morphisms (or at least objects) of $\T\Mod$ are convex and hence exponentiable, thus allowing us to establish useful sufficient conditions for the (local) cartesian closure of $\T\Mod$, and for $\T\Mod$ to be a quasitopos (Theorems \ref{lcc_thm} and \ref{quasitopos_thm}). In Examples \ref{preord_cc} and \ref{trans_egg} we provide two explanations for the cartesian closure of the categories $\Preord$ and $\Pos$ of preordered and partially ordered sets. 

In \S\ref{non_disc_section} we identify a useful sufficient condition for objects and morphisms of $\T\Mod$ to be exponentiable, which we again call \emph{convexity} (Definition \ref{schema_convex}), when $\T$ is a certain general kind of relational Horn theory over a preordered relational signature $\Pi$ that is (pointwise) a complete Heyting algebra. Again, when the axioms of $\T$ have certain properties (Definition \ref{schema_safe}), we can show that all morphisms (or at least objects) of $\T\Mod$ are exponentiable, thus allowing us to provide useful sufficient conditions for the (local) cartesian closure of $\T\Mod$, and for $\T\Mod$ to be a quasitopos (Theorems \ref{lcc_thm_V} and \ref{quasitopos_thm_2}). Our results in \S\ref{non_disc_section} recover a known sufficient condition for a morphism of $\V\Cat$ (i.e.~a $\V$-functor) to be exponentiable, where $\V$ is a commutative unital quantale whose underlying complete lattice is a complete Heyting algebra (see Example \ref{schema_convex_egg}; this condition is also \emph{necessary} for exponentiability, as shown in \cite[Theorem 3.4]{Exponentiation_in_V_categories}). In Remark \ref{further} we discuss some further questions that could be pursued.     

\section{General categorical background}
\label{background}

We first recall some general categorical background that we shall require.

\begin{para}
\label{lcc_para}
{\em
An object $C$ of a category $\C$ with finite products is \emph{exponentiable} if the functor $C \times (-) : \C \to \C$ has a right adjoint $(-)^C : \C \to \C$. A morphism $f : C \to D$ of a category $\C$ with finite limits is \emph{exponentiable} if it is exponentiable as an object of the slice category $\C/D$; equivalently, if the pullback functor $f^* : \C/D \to \C/C$ has a right adjoint (see \cite[Corollary 1.2]{Cartesianness}). A category $\C$ is \emph{cartesian closed} (resp.~\emph{locally cartesian closed}) if it has finite products (resp.~finite limits) and every object (resp.~morphism) of $\C$ is exponentiable. In particular, every locally cartesian closed category is cartesian closed (since an object $C$ is exponentiable iff the unique morphism $!_C : C \to 1$ is exponentiable).
}
\end{para}

\begin{para}
\label{partial_prod_para}
{\em
Let $\C$ be a category with finite limits, let $f : X \to Z$ be a morphism of $\C$, and let $Y$ be an object of $\C$. A \emph{partial product of $Y$ over $f$} \cite{Dyckhoff} is an object $P = P(Y, f)$ equipped with morphisms $p : P \to Z$ and $\varepsilon : P \times_{Z} X \to Y$ satisfying the universal property that for all morphisms $q : Q \to Z$ and $g : Q \times_{Z} X \to Y$, there is a unique morphism $h : Q \to P$ such that $p \circ h = q$ and $\varepsilon \circ \left(h \times_Z 1_X\right) = g$, as in the following commutative diagram:
% https://q.uiver.app/?q=WzAsNyxbMiwwLCJQIFxcdGltZXNfe1p9IFgiXSxbNCwwLCJYIl0sWzIsMiwiUCJdLFs0LDIsIloiXSxbMCwwLCJZIl0sWzEsMywiUSJdLFsxLDEsIlEgXFx0aW1lc197Wn0gWCJdLFswLDEsIlxccGlfMiJdLFswLDIsIlxccGlfMSIsMl0sWzIsMywicCJdLFsxLDMsImYiXSxbNCwwLCJcXHZhcmVwc2lsb24iLDAseyJzdHlsZSI6eyJ0YWlsIjp7Im5hbWUiOiJhcnJvd2hlYWQifSwiaGVhZCI6eyJuYW1lIjoibm9uZSJ9fX1dLFs1LDIsImgiLDEseyJzdHlsZSI6eyJib2R5Ijp7Im5hbWUiOiJkYXNoZWQifX19XSxbNSwzLCJxIiwyXSxbNiwwLCJoIFxcdGltZXNfe1p9IDFfWCIsMV0sWzYsNV0sWzYsMV0sWzYsNCwiZyJdXQ==
\[\begin{tikzcd}
	Y && {P \times_{Z} X} && X \\
	& {Q \times_{Z} X} \\
	&& P && Z \\
	& Q
	\arrow["{\pi_2}", from=1-3, to=1-5]
	\arrow["{\pi_1}"', from=1-3, to=3-3]
	\arrow["p", from=3-3, to=3-5]
	\arrow["f", from=1-5, to=3-5]
	\arrow["\varepsilon", tail reversed, no head, from=1-1, to=1-3]
	\arrow["h"{description}, dashed, from=4-2, to=3-3]
	\arrow["q"', from=4-2, to=3-5]
	\arrow["{h \times_{Z} 1_X}"{description}, from=2-2, to=1-3]
	\arrow[from=2-2, to=4-2]
	\arrow[from=2-2, to=1-5]
	\arrow["g", from=2-2, to=1-1]
\end{tikzcd}\]
We say that $\C$ \emph{has all partial products over $f$} if every object of $\C$ has a partial product over $f$.
By \cite[Lemma 2.1]{Dyckhoff}, a morphism $f$ of $\C$ is exponentiable iff $\C$ has all partial products over $f$. 
}
\end{para}

\begin{para}
\label{concrete_para}
{\em
A \emph{concrete category (over $\Set$)} is a category $\C$ equipped with a faithful functor $|-| : \C \to \Set$, and the \emph{fibre} of a set $S$ is the (preordered) class of objects $X$ of $\C$ with $|X| = S$ (see \cite[Definition 5.4]{AHS}). A concrete category $\C$ is \emph{fibre-small} if the fibre of each set is small (see \cite[Definition 5.4]{AHS}), and it is \emph{well-fibred} if it is fibre-small and every set with at most one element has exactly one element in its fibre (see \cite[Definition 27.20]{AHS}). A concrete category $\C$ \emph{admits constant morphisms} if for all objects $X$ and $Y$ of $\C$, every constant function $|X| \to |Y|$ lifts to a $\C$-morphism $X \to Y$.  
}
\end{para}

\begin{para}
\label{quasitopos_para}
{\em
Recall from (e.g.)~\cite[Definition 28.7]{AHS} that a category is a \emph{quasitopos} if it is finitely complete and finitely cocomplete, locally cartesian closed, and has a weak subobject classifier (i.e.~a classifier of strong subobjects; we shall not need an explicit definition). In particular, if $\C$ is a concrete category that is \emph{topological over $\Set$} (see \cite[\S 21]{AHS} for an explicit definition, which we also shall not need), then $\C$ is complete and cocomplete and has a weak subobject classifier by \cite[III.4.J]{Monoidaltop}, which is obtained by equipping the subobject classifier of $\Set$ (i.e.~any two-element set) with the indiscrete structure. So a topological category over $\Set$ is a quasitopos iff it is locally cartesian closed. 
}
\end{para}

\section{Background on categories of relational structures}
\label{first_section}

We now review relational Horn theories and their categories of models; much of the content of this section is taken from the author's work \cite[\S 3 and \S4]{Extensivity}; see also \cite[\S 3]{Monadsrelational}. 

\begin{defn}
\label{relational_sig}
{\em
A \textbf{relational signature} is a set $\Pi$ of \textbf{relation symbols} equipped with an assignment to each relation symbol of a \emph{finite arity}, i.e.~a natural number $n \geq 1$.
}
\end{defn}

\noindent We shall usually write $R$ for an arbitrary relation symbol. We fix a relational signature $\Pi$ for the rest of \S\ref{first_section}.

\begin{defn}
\label{edge}
{\em
Let $S$ be a set. A \textbf{$\Pi$-edge} in $S$ is a pair $(R, (s_1, \ldots, s_n))$ consisting of a relation symbol $R \in \Pi$ (of arity $n \geq 1$) and an $n$-tuple $(s_1, \ldots, s_n) \in S^n$. A \textbf{$\Pi$-structure} $X$ consists of a set $|X|$ equipped with a subset $R^X \subseteq |X|^n$ for each relation symbol $R \in \Pi$ (of arity $n \geq 1$). We can also describe a $\Pi$-structure $X$ as a set $|X|$ equipped with a set $\E(X)$ of $\Pi$-edges in $|X|$: if $R \in \Pi$ of arity $n \geq 1$, then $(x_1, \ldots, x_n) \in R^X$ iff $\E(X)$ contains the $\Pi$-edge $(R, (x_1, \ldots, x_n))$. We shall often write $X \models Rx_1\ldots x_n$ instead of $(x_1, \ldots, x_n) \in R^X$.
}
\end{defn} 

\noindent When $R$ is a binary relation symbol (i.e.~its arity is $2$), we shall also sometimes write $X \models x_1Rx_2$ rather than $X \models Rx_1x_2$.   

\begin{defn}
\label{Pi_morphism}
{\em
Let $h : S \to T$ be a function from a set $S$ to a set $T$, and let $e = (R, (s_1, \ldots, s_n))$ be a $\Pi$-edge in $S$. We write $h \cdot e = h \cdot (R, (s_1, \ldots, s_n))$ for the $\Pi$-edge $(R, (h(s_1), \ldots, h(s_n)))$ in $T$. For a set $E$ of $\Pi$-edges in $S$, we write $h \cdot E$ for the set of $\Pi$-edges $\{h \cdot e \mid e \in S\}$ in $T$. A \textbf{($\Pi$-)morphism $h : X \to Y$} from a $\Pi$-structure $X$ to a $\Pi$-structure $Y$ is a function $h : |X| \to |Y|$ satisfying $h \cdot \E(X) \subseteq \E(Y)$. We then have the concrete category $\Str(\Pi)$ of $\Pi$-structures and their morphisms.    
}
\end{defn} 

\noindent We now turn to the syntax of relational Horn theories. Throughout the paper, we fix an infinite set of variables $\Var$.  

\begin{defn}
\label{Horn_formula}
{\em
A \textbf{relational Horn formula (over $\Pi$)} is an expression $\Phi \Longrightarrow \psi$, where $\Phi$ is a set of $\Pi$-edges in $\Var$ and $\psi$ is a $\left(\Pi \cup \{=\}\right)$-edge in $\Var$, for a binary relation symbol $=$ not in $\Pi$. If $\Phi = \{\varphi_1, \ldots, \varphi_n\}$ is finite, then we write $\varphi_1, \ldots, \varphi_n \Longrightarrow \psi$, and if $\Phi = \varnothing$, then we write $\Longrightarrow \psi$. A \textbf{relational Horn formula without equality (over $\Pi$)} is a relational Horn formula $\Phi \Longrightarrow \psi$ (over $\Pi$) such that $\psi$ does not contain $=$.    
}
\end{defn}

\noindent We shall typically write $\Pi$-edges in $\Var$ as $Rv_1 \ldots v_n$\footnote{We emphasize that the notation $Rv_1 \ldots v_n$ is \emph{not} meant to suggest that the variables $v_1, \ldots, v_n$ are pairwise distinct; i.e.~we may have $v_i \equiv v_j$ for distinct $1 \leq i, j \leq n$.} rather than $(R, (v_1, \ldots, v_n))$, and when $R \in \Pi$ has arity $2$, we shall typically write $v_1 R v_2$ rather than $Rv_1v_2$.

\begin{defn}
\label{variable_set}
{\em
For any $(\Pi \cup \{=\})$-edge $\varphi \equiv Rv_1 \ldots v_n$ in $\Var$, we define $\Var(\varphi) := \{v_1, \ldots, v_n\}$. If $\Phi$ is a set of $(\Pi \cup \{=\})$-edges in $\Var$, then we set $\Var(\Phi) := \bigcup_{\varphi \in \Phi} \Var(\varphi)$.
}
\end{defn} 

\begin{defn}
\label{Horn_theory}
{\em
A \textbf{relational Horn theory $\T$ (without equality)} is a set of relational Horn formulas (without equality) over $\Pi$, which we call the \emph{axioms} of $\T$. We shall assume throughout that if $\Phi \Longrightarrow v_1 = v_2$ is an axiom of $\T$ with equality, then $\Var(\Phi) = \{v_1, v_2\}$\footnote{This mild but simplifying assumption (which is satisfied by all the examples of Example \ref{without_equality_examples}) is explicitly invoked in the proofs of Theorems \ref{convex_thm_unified} and \ref{schema_convex_part_prod}.}.   
}
\end{defn}

\begin{defn}
\label{formula_satis}
{\em
Let $X$ be a $\Pi$-structure. We define a $\left(\Pi \cup \{=\}\right)$-structure $\overline{X}$ by $\left|\overline{X}\right| := |X|$ and $\E\left(\overline{X}\right) := \E(X) \cup \{(=, (x, x)) \mid x \in |X|\}$. A \textbf{valuation in $X$} is a function $\kappa : \Var \to |X|$. We say that $X$ \textbf{satisfies} a relational Horn formula $\Phi \Longrightarrow \psi$ if $\Xbar \models \kappa \cdot \psi$ for every valuation $\kappa$ in $X$ such that $X \models \kappa \cdot \varphi$ for each $\varphi \in \Phi$. A \textbf{model} of a relational Horn theory $\T$ (or \textbf{$\T$-model}) is a $\Pi$-structure that satisfies all axioms of $\T$. We let $\T\Mod$ be the full subcategory of $\Str(\Pi)$ spanned by the $\T$-models, so that $\T\Mod$ is a concrete category.          
}
\end{defn}

\begin{rmk}
\label{val_rmk}
{\em
Let $X$ be a $\Pi$-structure, and let $\varphi$ be a ($\Pi \cup \{=\}$)-edge in $\Var$. It is clear that if $\kappa, \kappa' : \Var \to |X|$ are valuations that agree on $\Var(\varphi)$ but may not agree on variables in $\Var \setminus \Var(\varphi)$, then $X \models \kappa \cdot \varphi$ iff $X \models \kappa' \cdot \varphi$. Hence, we shall often specify valuations simply by defining their values on a specific subset of $\Var$ of interest, with the understanding that the valuation is defined arbitrarily on variables outside of this specific subset. 
}
\end{rmk}

\begin{egg}
\label{without_equality_examples}
{\em
We have the following central examples of relational Horn theories:
\begin{enumerate}[leftmargin=*]
\item If $\T$ is the empty relational Horn theory, then of course $\T\Mod = \Str(\Pi)$. In particular, if $\Pi$ is empty, then $\T\Mod = \Set$.\label{empty} 

\item Let $\Pi$ have a single binary relation symbol $\leq$, and let $\T$ be the relational Horn theory over $\Pi$ that contains the axioms $\Longrightarrow x \leq x$ and $x \leq y, y \leq z \Longrightarrow x \leq z$. Then $\T\Mod$ is the concrete category $\mathsf{Preord}$ of preordered sets and monotone functions. If one adds the additional axiom $x \leq y, y \leq x \Longrightarrow x = y$, then the category of models of the resulting relational Horn theory is the concrete category $\mathsf{Pos}$ of posets and monotone functions.\label{preord}  

\item The following examples derive from \cite{Kelly, Metagories}. Let $(\V, \leq, \tensor, \sfk)$ be a \emph{commutative unital quantale} \cite{Quantales}, i.e.~$(\V, \leq)$ is a complete lattice and $(\V, \tensor, \sfk)$ is a commutative monoid and $\tensor$ preserves all suprema in each variable. A \emph{$\V$-graph} or \emph{$\V$-valued relation} $(X, d)$ is a set $X$ equipped with a function $d : X \times X \to \V$. A \emph{reflexive $\V$-graph} is a $\V$-graph $(X, d)$ satisfying $d(x, x) \geq \sfk$ for all $x \in X$. A \emph{$\V$-category} is a reflexive $\V$-graph $(X, d)$ satisfying $d(x, z) \geq d(x, y) \tensor d(y, z)$ for all $x, y, z \in X$. A \emph{pseudo-$\V$-metric space} is a $\V$-category $(X, d)$ satisfying $d(x, y) = d(y, x)$ for all $x, y \in X$. Finally, a \emph{$\V$-metric space} is a pseudo-$\V$-metric space $(X, d)$ satisfying $d(x, y) \geq \sfk \Longrightarrow x = y$ for all $x, y \in X$. If $(X, d_X)$ and $(Y, d_Y)$ are $\V$-graphs, then a \emph{$\V$-functor} or \emph{$\V$-contraction} $h : (X, d_X) \to (Y, d_Y)$ is a function $h : X \to Y$ such that $d_X(x, x') \leq d_Y(h(x), h(x'))$ for all $x, x' \in X$. We let $\V\Gph$ be the concrete category of $\V$-graphs and $\V$-functors, and we let $\V\RGph$ (resp.~$\V\Cat$, $\PMet_\V$, $\Met_\V$) be the full subcategory of $\V\Gph$ consisting of the reflexive $\V$-graphs (resp.~the $\V$-categories, the pseudo-$\V$-metric spaces, the $\V$-metric spaces). 

Let $\Pi_\V$ have binary relation symbols $\sim_v$ for all $v \in \V$. We let $\T_{\V\Gph}$ be the relational Horn theory over $\Pi_\V$ that consists of the axioms $x \sim_v y \Longrightarrow x \sim_{v'} y$ for all $v, v' \in \V$ with $v \geq v'$, together with the axioms $\{x \sim_{v_i} y \mid i \in I\} \Longrightarrow x \sim_{\bigvee_i v_i} y$ for all small families $(v_i)_{i \in I}$ of elements of $\V$. We let $\T_{\V\RGph}$ be the relational Horn theory over $\Pi_\V$ that extends $\T_{\V\Gph}$ by adding the single axiom $\Longrightarrow x \sim_\sfk x$. We let $\T_{\V\Cat}$ be the relational Horn theory over $\Pi_\V$ that extends $\T_{\V\RGph}$ by adding the axioms $x \sim_v y, y \sim_{v'} z \Longrightarrow x \sim_{v \tensor v'} z$ for all $v, v' \in \V$. We let $\T_{\PMet_\V}$ be the relational Horn theory over $\Pi_\V$ that extends $\T_{\V\Cat}$ by adding the axioms $x \sim_v y \Longrightarrow y \sim_v x$ for all $v \in \V$. Finally, we let $\T_{\Met_\V}$ be the relational Horn theory over $\Pi_\V$ that extends $\T_{\PMet_\V}$ by adding the single axiom $x \sim_\sfk y \Longrightarrow x = y$. It is shown in \cite[Appendix]{Extensivity} that $\T_{\V\Gph}\Mod$ (resp.~$\T_{\V\RGph}\Mod$, $\T_{\V\Cat}\Mod$, $\T_{\PMet_\V}\Mod$, $\T_{\Met_\V}\Mod$) is isomorphic to $\V\Gph$ (resp.~$\V\RGph$, $\V\Cat$, $\PMet_\V$, $\Met_\V$)\footnote{The first two isomorphisms are not explicitly established in \cite[Appendix]{Extensivity}, but they immediately follow from the proofs given there.}.\label{quantale}
\end{enumerate}
}
\end{egg}

\begin{para}
\label{relational_top}
{\em
Let $\T$ be a relational Horn theory. If $\T$ is without equality, then the concrete category $\T\Mod$ is topological over $\Set$ (see \cite[Proposition 4.4]{Extensivity} or \cite[Proposition 5.1]{Rosickyconcrete}. In general, the concrete category $\T\Mod$ is \emph{monotopological} over $\Set$ (in the sense of \cite[Definition 21.38]{AHS}; see \cite[Proposition 5.5]{Rosickyconcrete}). Thus, given a small diagram $D : \A \to \T\Mod$, the limit cone of $D$ is the \emph{initial lift} of the limit cone of $|-| \circ D$ in $\Set$ (see e.g.~\cite[21.15]{AHS}). In particular, the functor $|-| : \T\Mod \to \Set$ strictly preserves small limits. The category $\T\Mod$ is also cocomplete, and moreover locally presentable (by \cite[Proposition 5.30]{LPAC}). Finally, the full subcategory $\T\Mod \hookrightarrow \Str(\Pi)$ is (epi-)reflective by \cite[Proposition 3.6]{Monadsrelational}, so that every $\Pi$-structure generates a free $\T$-model. 
}
\end{para}

\begin{para}
\label{relational_prod}
{\em
Let $\T$ be a relational Horn theory. In view of \ref{relational_top}, the product $X \times Y$ in $\T\Mod$ of $\T$-models $X$ and $Y$ is given by $|X \times Y| = |X| \times |Y|$ with 
\[ R^{X \times Y} = \left\{\left((x_1, y_1), \ldots, (x_n, y_n)\right) \in \left(|X| \times |Y|\right)^n \mid X \models Rx_1\ldots x_n \text{ and } Y \models Ry_1\ldots y_n\right\} \]
for each $R \in \Pi$ of arity $n \geq 1$, with the product projections as in $\Set$. The terminal object $1$ of $\T\Mod$ is given by $|1| = \{\ast\}$ and $1 \models R\ast\ldots\ast$ for each $R \in \Pi$.

The pullback $A \times_C B$ in $\T\Mod$ of $\T$-model morphisms $f : A \to C$ and $g : B \to C$ is given by $\left|A \times_C B\right| = |A| \times_{|C|} |B| = \{(a, b) \in |A| \times |B| \mid f(a) = g(b)\}$ with
\[ R^{A \times_C B} = \left\{((a_1, b_1), \ldots, (a_n, b_n)) \in \left|A \times_C B\right|^n \mid A \models R a_1 \ldots a_n \text{ and } B \models R b_1 \ldots b_n\right\} \] for each $R \in \Pi$ of arity $n \geq 1$, with the pullback projections as in $\Set$.  
}
\end{para}

\begin{para}
\label{closed_structure}
{\em
Let $\T$ be a relational Horn theory. For the remainder of the paper we shall be concerned with providing sufficient conditions for $\T\Mod$ to be (locally) cartesian closed, but it is worth noting that $\T\Mod$ is always at least \emph{symmetric monoidal closed}, which we now recall from \cite[Definition 3.11 and Corollary 3.13]{Monadsrelational}. For $\T$-models $X$ and $Y$, the internal hom $[X, Y]$ has underlying set $\left|[X, Y]\right| = \T\Mod(X, Y)$, and for $R \in \Pi$ of arity $n \geq 1$ and $\Pi$-morphisms $f_1, \ldots, f_n : X \to Y$ we have $[X, Y] \models Rf_1\ldots f_n$ iff $Y \models Rf_1(x)\ldots f_n(x)$ for each $x \in |X|$. For $\T$-models $X$ and $Y$, the tensor product $X \tensor Y$ is the free $\T$-model \eqref{relational_top} on the $\Pi$-structure $A$ with $|A| := |X| \times |Y|$ and $A \models R(x_1, y_1)\ldots(x_n, y_n)$ iff (i) $x_1 = \ldots = x_n$ and $Y \models R y_1 \ldots y_n$, or (ii) $y_1 = \ldots = y_n$ and $X \models R x_1 \ldots x_n$, for each $R \in \Pi$ of arity $n \geq 1$. The tensor unit is the free $\T$-model \eqref{relational_top} on the $\Pi$-structure $I$ with $|I|$ a singleton set and $R^I := \varnothing$ for each $R \in \Pi$.     
} 
\end{para}

\section{Exponentiability in $\Str(\Pi)$}
\label{exp_Str}

We fix a relational signature $\Pi$ for the remainder of this section. We shall first consider exponentiability in the category $\Str(\Pi)$ of $\Pi$-structures and $\Pi$-morphisms; in fact, we shall prove in Theorem \ref{Str_lcc} that $\Str(\Pi)$ is always locally cartesian closed, and even a quasitopos.

\begin{defn}
\label{restricted_structure}
{\em
Let $f : X \to Z$ be a morphism of $\Pi$-structures. For each $z \in |Z|$, we define a $\Pi$-structure $X_{f, z}$ as follows: we set $\left| X_{f, z}\right| := f^{-1}(z) \subseteq |X|$, and for each $R \in \Pi$ of arity $n \geq 1$, we set $R^{X_{f, z}} := R^X \cap \left|X_{f, z}\right|^n$.
}
\end{defn}

\begin{defn}
\label{partial_product_Str}
{\em
Let $f : X \to Z$ be a morphism of $\Pi$-structures, and let $Y$ be a $\Pi$-structure. We define a $\Pi$-structure $P = P(Y, f)$ as follows. We set
\[ |P| := \left\{(j, z) \mid z \in |Z| \text{ and } j \in \Set\left(|X_{f, z}|, |Y|\right)\right\}. \] Now let $R \in \Pi$ of arity $n \geq 1$, and let $(j_1, z_1), \ldots, (j_n, z_n) \in |P|$. Then we set
\[ P \models R(j_1, z_1) \ldots (j_n, z_n) \] iff $Z \models Rz_1 \ldots z_n$ and for all $x_1 \in f^{-1}(z_1), \ldots, x_n \in f^{-1}(z_n)$ such that $X \models Rx_1 \ldots x_n$, we have $Y \models Rj_1(x_1)\ldots j_n(x_n)$. We then have a $\Pi$-morphism $p : P \to Z$ given by $p(j, z) := z$ for each $(j, z) \in |P|$, and a $\Pi$-morphism $\varepsilon : P \times_Z X \to Y$ given by $\varepsilon((j, z), x) := j(x)$ for each $((j, z), x) \in |P \times_Z X|$. 
}
\end{defn} 

\begin{rmk}
\label{Pi_partial_rmk}
{\em
In the definition of $|P|$ in Definition \ref{partial_product_Str}, one might wonder why the first component of an element $(j, z) \in |P|$ is just a function $j : |X_{f, z}| \to |Y|$ rather than a $\Pi$-morphism $j : X_{f, z} \to Y$. The forgetful functor $|-| : \Str(\Pi) \to \Set$ is represented (not by the terminal object $1$ but) by the tensor unit $\Pi$-structure $I$ defined in \ref{closed_structure} by $|I| := \{\ast\}$ and $\E(I) := \varnothing$. In other words, for each $\Pi$-structure $X$ we have a natural bijection $|X| \cong \Str(\Pi)(I, X)$. If we want $P$ to be a partial product of $Y$ over $f$ (see \ref{partial_prod_para}), then by setting $Q := I$ in the definition of partial product, we see that $|P| \cong \Str(\Pi)(I, P)$ must be isomorphic to the set given in Definition \ref{partial_product_Str}, since for each $z \in |Z|$ and corresponding $\Pi$-morphism $\bar{z} : I \to Z$, the pullback $I \times_Z X$ satisfies $|I \times_Z X| \cong |X_{f, z}|$ but $\E\left(I \times_Z X\right) = \varnothing$, so that a $\Pi$-morphism $I \times_Z X \to Y$ is just a function $|X_{f, z}| \to |Y|$. However, when we consider $\T\Mod$ for a (\emph{reflexive}, \ref{reflexive}) relational Horn theory $\T$ in \S\ref{exponentiable_Tmod} below, we shall need to define $|P|$ in (perhaps) the more expected way, with the first component of $(j, z) \in |P|$ being a $\Pi$-morphism $j : X_{f, z} \to Y$ (see Definition \ref{partial_product_unified} and Remark \ref{Tmod_partial_rmk}). 
}
\end{rmk}

\begin{prop}
\label{partial_product_prop_Str}
Let $f : X \to Z$ be a morphism of $\Pi$-structures, and let $Y$ be a $\Pi$-structure. Then the $\Pi$-structure $P = P(Y, f)$ of \eqref{partial_product_Str} is a partial product of $Y$ over $f$ in $\Str(\Pi)$.  
\end{prop}

\begin{proof}
We defined the required $\Pi$-morphisms $p : P \to Z$ and $\varepsilon : P \times_Z X \to Y$ in Definition \ref{partial_product_Str}. So let $q : Q \to Z$ and $g : Q \times_Z X \to Y$ be morphisms of $\Str(\Pi)$. We must show that there is a unique $\Pi$-morphism $h : Q \to P$ satisfying $p \circ h = q$ and $\varepsilon \circ \left(h \times_Z 1_X\right) = g$. So let $a \in |Q|$, and let us define $h(a) := (j_a, q(a)) \in |P|$. We have $q(a) \in |Z|$, and we must define a function $j_a : |X_{f, q(a)}| \to |Y|$. For each $x \in |X_{f, q(a)}| = f^{-1}(q(a))$ we have $f(x) = q(a)$, so that $(a, x) \in |Q \times_Z X|$, and we can then set $j_a(x) := g(a, x) \in |Y|$. We must now show that the function $h : |Q| \to |P|$ is a $\Pi$-morphism $h : Q \to P$. So let $R \in \Pi$ of arity $n \geq 1$, let $a_1, \ldots, a_n \in |Q|$, and suppose that $Q \models Ra_1 \ldots a_n$; we must show that $P \models Rh(a_1)\ldots h(a_n)$, i.e.~that $P \models R(j_{a_1}, q(a_1))\ldots(j_{a_n}, q(a_n))$. We first have $Z \models Rq(a_1)\ldots q(a_n)$ because $Q \models Ra_1 \ldots a_n$ and $q : Q \to Z$ is a $\Pi$-morphism. Now let $x_i \in f^{-1}(q(a_i))$ for each $1 \leq i \leq n$, suppose that $X \models Rx_1 \ldots x_n$, and let us show that $Y \models Rj_{a_1}(x_1)\ldots j_{a_n}(x_n)$, i.e.~that $Y \models Rg(a_1, x_1)\ldots g(a_n, x_n)$. But we have $Q \models Ra_1 \ldots a_n$ and thus $Q \times_Z X \models R(a_1, x_1)\ldots (a_n, x_n)$, and $g : Q \times_Z X \to Y$ is a $\Pi$-morphism. So $h : Q \to P$ is a $\Pi$-morphism, and we clearly have $p \circ h = q$ and $\varepsilon \circ \left(h \times_Z 1_X\right) = g$. 

To show the uniqueness of $h$, let $k : Q \to P$ be any $\Pi$-morphism satisfying $p \circ k = q$ and $\varepsilon \circ \left(k \times_Z 1_X\right) = g$, and let us show for each $a \in |Q|$ that $k(a) = h(a) = (j_a, q(a))$. Since $p \circ k = q$, we just need to show that $\pi_1(k(a)) = j_a : |X_{f, q(a)}| \to |Y|$, i.e.~that $\pi_1(k(a))(x) = j_a(x) = g(a, x)$ for each $x \in |X_{f, q(a)}|$; but this is true because $\varepsilon \circ \left(k \times_Z 1_X\right) = g$.      
\end{proof}

\noindent If $\Pi$ just contains a single binary relation symbol, then it is known that $\Str(\Pi)$ is a quasitopos (see e.g.~\cite[Examples 28.9(2)]{AHS}). We can now extend this result to $\Str(\Pi)$ for an arbitrary relational signature $\Pi$.

\begin{theo}
\label{Str_lcc}
$\Str(\Pi)$ is a quasitopos for every relational signature $\Pi$.
\end{theo}

\begin{proof}
$\Str(\Pi)$ is locally cartesian closed by \ref{partial_prod_para} and Proposition \ref{partial_product_prop_Str}. Since the concrete category $\Str(\Pi)$ is topological over $\Set$ (by \ref{relational_top}, in view of Example \ref{without_equality_examples}.\ref{empty}), we then deduce from \ref{quasitopos_para} that $\Str(\Pi)$ is a quasitopos. 
\end{proof}

\section{Exponentiability in $\T\Mod$ for a relational Horn theory $\T$}
\label{exponentiable_Tmod}

We now consider exponentiability in $\T\Mod$ for a relational Horn theory $\T$. Unlike the situation for $\Str(\Pi)$ (see Theorem \ref{Str_lcc}), it is well known that $\T\Mod$ is in general not even cartesian closed, let alone locally cartesian closed or a quasitopos. For example (see Example \ref{without_equality_examples}), $\Preord$ and $\Pos$ are not locally cartesian closed, while the category $\Met$ of (extended) metric spaces (which is $\Met_\V$ for the \emph{Lawvere quantale} $\V$, see \cite[Example 3.7]{Extensivity}) is not even cartesian closed.    

\begin{defn}
\label{reflexive}
{\em
Let $\Pi$ be a relational signature. A $\Pi$-structure $X$ is \textbf{reflexive} if for each $R \in \Pi$ the relation $R^X$ is reflexive, i.e.~for each $x \in |X|$ we have $X \models R x \ldots x$. A relational Horn theory $\T$ is \textbf{reflexive} if every $\T$-model is reflexive.
}
\end{defn}

\begin{para}
\label{refl_para}
{\em
If $\T$ is a reflexive relational Horn theory, then $\T\Mod$ admits constant morphisms \eqref{concrete_para}. For if $X$ and $Y$ are $\T$-models and $h : |X| \to |Y|$ is a constant function, then $h$ is a $\Pi$-morphism $h : X \to Y$, because if $R \in \Pi$ of arity $n \geq 1$ and $X \models R x_1 \ldots x_n$, then $Y \models R h(x_1) \ldots h(x_n)$ because $h(x_1) = \ldots = h(x_n)$ and $Y$ is reflexive. More generally, any constant function from a $\Pi$-structure into a reflexive $\Pi$-structure is a $\Pi$-morphism. If $\T$ is reflexive, then $\T\Mod$ is clearly well-fibred \eqref{concrete_para}. 
}
\end{para}

\begin{defn}
\label{preord_sig}
{\em
Let $\Pi$ be a relational signature. We say that $\Pi$ is a \textbf{preordered relational signature} if for each $n \geq 1$, the set $\Pi(n)$ of relation symbols of arity $n$ is equipped with a preorder $\leq$ (i.e.~a reflexive and transitive binary relation). We say that a preordered relational signature $\Pi$ is \textbf{discrete} if each preordered set $\Pi(n)$ ($n \geq 1$) discrete, and we say that a preordered relational signature $\Pi$ is a \textbf{complete Heyting algebra} if each preordered set $\Pi(n)$ ($n \geq 1$) is a complete Heyting algebra (i.e.~a complete lattice in which binary meets distribute in each variable over arbitrary joins).
}
\end{defn}

\begin{egg}
\label{preord_sig_egg}
{\em
We consider the relational signatures of Examples \ref{without_equality_examples}.\ref{empty} and \ref{without_equality_examples}.\ref{preord} to be discrete, while if $(\V, \leq, \tensor, \sfk)$ is a commutative unital quantale with the associated relational signature $\Pi_\V$ of Example \ref{without_equality_examples}.\ref{quantale}, then we equip the set $\Pi_\V(2)$ with a preorder that is generally \emph{not} discrete: for $v, v' \in \V$, we set $\sim_v \ \leq \ \sim_{v'}$ iff $v \leq v'$. If $(\V, \leq)$ is a complete Heyting algebra, then the relational signature $\Pi_\V$ is a complete Heyting algebra.
}
\end{egg}

\noindent We fix a preordered relational signature $\Pi$ for the remainder of this section.   

\begin{defn}
\label{T_Pi}
{\em
We let $\T_\Pi$ be the relational Horn theory (without equality) over $\Pi$ that consists of the axioms $\Longrightarrow Rv\ldots v$ for all $R \in \Pi$, as well as the axioms $Rv_1\ldots v_n \Longrightarrow S v_1 \ldots v_n$ for all $R, S \in \Pi(n)$ ($n \geq 1$) such that $R \geq S$, where $v_1, \ldots, v_n$ are pairwise distinct variables. If $\Pi$ is a complete Heyting algebra, then we also stipulate that $\T_\Pi$ contains the axiom $\left\{R_i v_1 \ldots v_n \mid i \in I\right\} \Longrightarrow \left(\bigvee_i R_i\right)v_1 \ldots v_n$ for each $n \geq 1$ and small family $(R_i)_{i \in I}$ in $\Pi(n)$, where $v_1, \ldots, v_n$ are again pairwise distinct variables.

In particular, the relational Horn theory $\T_\Pi$ is reflexive \eqref{reflexive}. A model of $\T_\Pi$ is a reflexive $\Pi$-structure $X$ such that $R^X \subseteq S^X$ for all $R, S \in \Pi(n)$ ($n \geq 1$) with $R \geq S$; if $\Pi$ is a complete Heyting algebra, then also $\bigcap_i R_i^X \subseteq \left(\bigvee_i R_i\right)^X$ for each $n \geq 1$ and $(R_i)_{i \in I}$ in $\Pi(n)$.
}
\end{defn} 

\begin{defn}
\label{partial_product_unified}
{\em
Let $f : X \to Z$ be a morphism of $\T_\Pi$-models, and let $Y$ be a $\T_\Pi$-model. We define a $\Pi$-structure $P = P(Y, f)$ as follows. We set
\[ |P| := \left\{(j, z) \mid z \in |Z| \text{ and } j \in \Str(\Pi)\left(X_{f, z}, Y\right)\right\}. \] Now let $R \in \Pi(n)$ ($n \geq 1$) and let $(j_1, z_1), \ldots, (j_n, z_n) \in |P|$. Then we set
\[ P \models R(j_1, z_1) \ldots (j_n, z_n) \] iff $Z \models Rz_1 \ldots z_n$ and for all $x_1 \in f^{-1}(z_1), \ldots, x_n \in f^{-1}(z_n)$ and $S \in \Pi(n)$ with $R \geq S$ and $X \models Sx_1 \ldots x_n$, we have $Y \models Sj_1(x_1)\ldots j_n(x_n)$. We then have a $\Pi$-morphism $p : P \to Z$ defined by $p(j, z) := z$ for each $(j, z) \in |P|$, and a $\Pi$-morphism $\varepsilon : P \times_Z X \to Y$ defined by $\varepsilon((j, z), x) := j(x)$ for each $((j, z), x) \in |P \times_Z X|$. 
}
\end{defn}

\begin{rmk}
\label{Tmod_partial_rmk}
{\em
Since $\T_\Pi$ is reflexive, it follows that the forgetful functor $|-| : \T_\Pi\Mod \to \Set$ is represented by the terminal object $1$ \eqref{relational_prod}, so that for each $\T_\Pi$-model $X$ we have a natural bijection $|X| \cong \Str(\Pi)(1, X)$. If we want $P$ (in Definition \ref{partial_product_unified}) to be a partial product of $Y$ over $f$, then by setting $Q := 1$ in the definition of partial product, we see that $|P| \cong \Str(\Pi)(1, P)$ must be isomorphic to the set given in Definition \ref{partial_product_unified} (compare the situation for $\Str(\Pi)$ in Remark \ref{Pi_partial_rmk}).
}
\end{rmk} 

\noindent Without any assumption on $f : X \to Z$, for each $\T_\Pi$-model $Y$ the $\Pi$-structure $P(Y, f)$ of Definition \ref{partial_product_unified} is automatically a model of $\T_\Pi$:

\begin{prop}
\label{partial_product_is_model}
Let $f : X \to Z$ be a morphism of $\T_\Pi$-models, and let $Y$ be a $\T_\Pi$-model. Then the $\Pi$-structure $P = P(Y, f)$ of \eqref{partial_product_unified} is a model of $\T_\Pi$. 
\end{prop}

\begin{proof}
To verify that $P$ is reflexive, let $R \in \Pi(n)$ ($n \geq 1$) and let $(j, z) \in |P|$; we must show that $P \models R(j, z)\ldots(j, z)$. Since the $\T_\Pi$-model $Z$ is reflexive, we have $Z \models Rz\ldots z$. Now let $x_1, \ldots, x_n \in f^{-1}(z)$, and let $S \in \Pi(n)$ with $R \geq S$ and $X \models Sx_1 \ldots x_n$. Because $j : X_{f, z} \to Y$ is a $\Pi$-morphism and $X_{f, z} \models S x_1\ldots x_n$, we then have $Y \models Sj(x_1)\ldots j(x_n)$. 

Now let $R, S \in \Pi(n)$ ($n \geq 1$) with $R \geq S$, and suppose that $P \models R(j_1, z_1)\ldots (j_n, z_n)$; we must show that $P \models S(j_1, z_1) \ldots (j_n, z_n)$. We first have $Z \models S z_1 \ldots z_n$ because $Z$ is a $\T_\Pi$-model and $Z \models R z_1 \ldots z_n$. Now let $x_i \in f^{-1}(z_i)$ for each $1 \leq i \leq n$, let $T \in \Pi(n)$ with $S \geq T$ and $X \models Tx_1 \ldots x_n$, and let us show that $Y \models Tj_1(x_1) \ldots j_n(x_n)$. From $R \geq S$ and $S \geq T$ we obtain $R \geq T$ by transitivity of $\Pi(n)$, and then from $P \models R(j_1, z_1)\ldots(j_n, z_n)$ we obtain $Y \models Tj_1(x_1)\ldots j_n(x_n)$, as desired.

Suppose finally that $\Pi$ is a complete Heyting algebra, let $n \geq 1$ and let $(R_i)_{i \in I}$ be a small family in $\Pi(n)$, and suppose that $P \models R_i(j_1, z_1)\ldots(j_n, z_n)$ for each $i \in I$. To show that $P \models \left(\bigvee_i R_i\right)(j_1, z_1)\ldots(j_n, z_n)$, we first have $Z \models \left(\bigvee_i R_i\right)z_1\ldots z_n$ because $Z \models R_i z_1 \ldots z_n$ for each $i \in I$ and $Z$ is a $\T_\Pi$-model. Now let $x_i \in f^{-1}(z_i)$ for each $1 \leq i \leq n$, let $S \in \Pi(n)$ with $\bigvee_i R_i \geq S$ and $X \models Sx_1 \ldots x_n$, and let us show that $Y \models Sj_1(x_1)\ldots j_n(x_n)$. We have $S = S \wedge \bigvee_i R_i = \bigvee_i (S \wedge R_i)$ because $\Pi(n)$ is a complete Heyting algebra. Because $Y$ is a $\T_\Pi$-model, it then suffices to show that $Y \models (S \wedge R_i)j_1(x_1) \ldots j_n(x_n)$ for each $i \in I$, which readily follows from $P \models R_i(j_1, z_1)\ldots(j_n, z_n)$ and the fact that $X$ is a $\T_\Pi$-model.           
\end{proof}

\noindent We say that a relational Horn theory $\T$ is an \emph{extension of $\T_\Pi$ over $\Pi$} if $\T$ is a relational Horn theory over $\Pi$ whose set of axioms contains the axioms of $\T_\Pi$ (so that each model of $\T$ is a model of $\T_\Pi$). 

\begin{prop}
\label{partial_product_prop_unified}
Let $\T$ be any extension of $\T_\Pi$ over $\Pi$, let $f : X \to Z$ be a morphism of $\T\Mod$, let $Y$ be a $\T$-model, and suppose that the $\Pi$-structure $P = P(Y, f)$ of \eqref{partial_product_unified} is a $\T$-model. Then $P$ is a partial product of $Y$ over $f$ in $\T\Mod$.  
\end{prop}

\begin{proof}
We defined the required $\Pi$-morphisms $p : P \to Z$ and $\varepsilon : P \times_Z X \to Y$ in Definition \ref{partial_product_unified}. The proof is now almost identical to that of Proposition \ref{partial_product_prop_Str}. Using the notation of that proof, given morphisms $q : Q \to Z$ and $g : Q \times_Z X \to Y$ of $\T\Mod$, we must first show for each $a \in |Q|$ that the function $j_a : |X_{f, q(a)}| \to |Y|$ defined by $j_a(x) := g(a, x)$ for each $x \in f^{-1}(q(a))$ is a $\Pi$-morphism $j_a : X_{f, q(a)} \to Y$. So let $R \in \Pi$ of arity $n \geq 1$, let $x_1, \ldots, x_n \in f^{-1}(q(a))$, and suppose that $X_{f, q(a)} \models Rx_1 \ldots x_n$. Then $f(x_i) = q(a)$ and hence $(a, x_i) \in |Q \times_Z X|$ for each $1 \leq i \leq n$, and moreover $X \models Rx_1 \ldots x_n$. Since $Q \models Ra\ldots a$ by reflexivity of $\T_\Pi$, we then obtain $Q \times_Z X \models R(a, x_1)\ldots(a, x_n)$, and then because $g : Q \times_Z X \to Y$ is a $\Pi$-morphism, we deduce that $Y \models Rg(a, x_1)\ldots g(a, x_n)$, i.e.~that $Y \models Rj_a(x_1)\ldots j_a(x_n)$, as desired. 

We must also show that the function $h : |Q| \to |P|$ defined by $h(a) := (j_a, q(a))$ for each $a \in |Q|$ is a $\Pi$-morphism $h : Q \to P$. So let $R \in \Pi$ of arity $n \geq 1$, let $a_1, \ldots, a_n \in |Q|$, and suppose that $Q \models Ra_1 \ldots a_n$; we must show that $P \models Rh(a_1)\ldots h(a_n)$, i.e.~that $P \models R(j_{a_1}, q(a_1))\ldots(j_{a_n}, q(a_n))$. We first have $Z \models Rq(a_1)\ldots q(a_n)$ because $Q \models Ra_1 \ldots a_n$ and $q : Q \to Z$ is a $\Pi$-morphism. Now let $x_i \in f^{-1}(q(a_i))$ for each $1 \leq i \leq n$, let $S \in \Pi(n)$ satisfy $R \geq S$ and $X \models Sx_1 \ldots x_n$, and let us show that $Y \models Sj_{a_1}(x_1)\ldots j_{a_n}(x_n)$, i.e.~that $Y \models Sg(a_1, x_1)\ldots g(a_n, x_n)$. But we have $Q \models Sa_1 \ldots a_n$ (because $Q$ is a $\T_\Pi$-model) and thus $Q \times_Z X \models S(a_1, x_1)\ldots (a_n, x_n)$, and so the result follows because $g : Q \times_Z X \to Y$ is a $\Pi$-morphism.      
\end{proof}

For an extension $\T$ of $\T_\Pi$ over $\Pi$ and a morphism $f : X \to Z$ of $\T\Mod$, Proposition \ref{partial_product_prop_unified} entails (in view of \ref{partial_prod_para}) that $f$ will be exponentiable if, for each $\T$-model $Y$, the $\Pi$-structure $P(Y, f)$ of Definition \ref{partial_product_unified} is a model of $\T$ (note that it is already a model of $\T_\Pi$ by Proposition \ref{partial_product_is_model}). In \S\ref{disc_section} and \S\ref{non_disc_section} we shall turn to identifying a sufficient condition on $f$, which we call \emph{convexity}, that will entail this.  

\section{Convexity in the discrete case}
\label{disc_section}

We shall first suppose (throughout this section) that $\T$ is a reflexive relational Horn theory \eqref{reflexive} over a \emph{discrete} relational signature $\Pi$; in \S\ref{non_disc_section} we shall consider the case where $\Pi$ is not necessarily discrete. Note that when $\Pi$ is discrete, then we can just take the axioms of $\T_\Pi$ to be $\Longrightarrow Rv\ldots v$ for all $R \in \Pi$, since the axiom $Rv_1 \ldots v_n \Longrightarrow Rv_1 \ldots v_n$ (for $n \geq 1$ and $R \in \Pi(n)$) is (of course) automatically satisfied by every $\Pi$-structure. Thus, a relational Horn theory $\T$ over $\Pi$ is an extension of $\T_\Pi$ iff it is reflexive\footnote{Technically, if $\T$ is reflexive, then the relational Horn formulas $\Longrightarrow Rv\ldots v$ for $R \in \Pi$ need not be axioms of $\T$, so that $\T$ need not be an extension of $\T_\Pi$ as we have defined this concept; but we can clearly assume w.l.o.g.~that these relational Horn formulas \emph{are} axioms of $\T$.}. 

For a reflexive relational Horn theory $\T$ over the discrete relational signature $\Pi$, we now identify a useful sufficient condition for a morphism of $\T\Mod$ to be exponentiable (see Theorem \ref{convex_exp_unified}). We write $\T\setminus \T_\Pi$ for the relational Horn theory over $\Pi$ whose axioms are the axioms of $\T$ that do not belong to $\T_\Pi$, i.e.~the axioms of $\T$ other than the reflexivity axioms $\Longrightarrow Rv\ldots v$ ($R \in \Pi$).

\begin{defn}
\label{convex_unified}
{\em
Let $f : X \to Z$ be a morphism of $\T\Mod$, and let $\Phi \Longrightarrow Rv_1 \ldots v_n$ be an axiom of $\T\setminus \T_\Pi$ without equality. We say that $f$ is \textbf{convex with respect to (the axiom) $\Phi \Longrightarrow Rv_1\ldots v_n$} if $f$ satisfies the following condition: 

Let $\kappa_Z : \Var \to |Z|$ be a valuation such that $Z \models \kappa_Z \cdot \varphi$ for each $\varphi \in \Phi$. Let $x_i \in f^{-1}(\kappa_Z(v_i))$ for each $1 \leq i \leq n$, and suppose that $X \models Rx_1 \ldots x_n$. Then there is a valuation $\kappa : \Var \to |X|$ such that $\kappa(v_i) = x_i$ for each $1 \leq i \leq n$ and $\kappa(v) \in f^{-1}(\kappa_Z(v))$ for each $v \in \Var(\Phi)\setminus\{v_1, \ldots, v_n\}$ and $X \models \kappa \cdot \varphi$ for each $\varphi \in \Phi$. 

We say that $f$ is \textbf{convex} if it is convex with respect to each axiom of $\T\setminus \T_\Pi$ without equality.
}  
\end{defn}

\begin{egg}
\label{convex_egg}
{\em
Let $\T$ be the reflexive relational Horn theory for preordered (partially ordered) sets (see Example \ref{without_equality_examples}.\ref{preord}), and let $f : X \to Z$ be a morphism of $\T\Mod = \Preord$ ($\T\Mod = \Pos$), i.e.~a monotone function between preordered (partially ordered) sets. The only axiom of $\T\setminus \T_\Pi$ without equality is the transitivity axiom $x \leq y, y \leq z \Longrightarrow x \leq z$, and one readily sees that $f$ is convex (with respect to this axiom) iff whenever we have $x_1, x_3 \in |X|$ and $z_2 \in |Z|$ satisfying $x_1 \leq x_3$ and $f(x_1) \leq z_2 \leq f(x_3)$, there is some $x_2 \in f^{-1}(z_2)$ such that $x_1 \leq x_2 \leq x_3$. So $f$ is convex iff $f$ is an \emph{interpolation-lifting map} in the sense of \cite[Definition 2.1]{Exp_morphisms}.  
}
\end{egg}

\begin{rmk}
\label{lifting_rmk}
{\em
The notion of convexity can be understood in terms of certain \emph{lifting properties} as follows. Let $\Phi$ be a set of $\Pi$-edges in $\Var$. We let $\Phi_\T$ be the free $\T$-model \eqref{relational_top} on the $\Pi$-structure $\Phi_\Pi$ defined by $\left|\Phi_\Pi\right| := \Var(\Phi)$ and $\E\left(\Phi_\Pi\right) := \Phi$, so that $\Phi_\T \models \varphi$ for each $\varphi \in \Phi$. For each $\T$-model $X$, we have that $\Pi$-morphisms $\Phi_\T \to X$ are in natural bijective correspondence with $\Pi$-morphisms $\Phi_\Pi \to X$, which in turn are in natural bijective correspondence with functions $\kappa : \Var(\Phi) \to |X|$ satisfying $X \models \kappa \cdot \varphi$ for each $\varphi \in \Phi$. In particular, for any relation symbol $R \in \Pi(n)$ ($n \geq 1$) and pairwise distinct variables $v_1, \ldots, v_n \in \Var$, we write $R_\T := \{(R, (v_1, \ldots, v_n))\}_\T$, and $\Pi$-morphisms $R_\T \to X$ are then in natural bijective correspondence with $n$-tuples $(x_1, \ldots, x_n) \in |X|^n$ such that $X \models Rx_1\ldots x_n$. 

Now let $f : X \to Z$ be a morphism of $\T\Mod$, and let $\Phi \Longrightarrow Rv_1\ldots v_n$ be an axiom of $\T \setminus \T_\Pi$ without equality. We write $\left(\Phi \Longrightarrow Rv_1\ldots v_n\right)_\T$ for the $\T$-model $\left(\Phi \cup \{(R, (v_1, \ldots, v_n))\}\right)_\T$. Since $\left(\Phi \Longrightarrow Rv_1\ldots v_n\right)_\T \models Rv_1\ldots v_n$, there is a corresponding canonical $\Pi$-morphism $f_{\Phi, R} : R_\T \to \left(\Phi \Longrightarrow Rv_1\ldots v_n\right)_\T$. Then it readily follows that $f$ is convex with respect to the axiom $\Phi \Longrightarrow Rv_1\ldots v_n$ iff for each outer commutative square of the form 
% https://q.uiver.app/?q=WzAsNCxbMCwwLCJTX1xcVCJdLFsyLDAsIlgiXSxbMiwyLCJaIl0sWzAsMiwiXFxQaGlfXFxUIl0sWzAsMV0sWzEsMiwiZiJdLFswLDMsImZfe1xcUGhpLCBTfSIsMl0sWzMsMl0sWzMsMSwiIiwxLHsic3R5bGUiOnsiYm9keSI6eyJuYW1lIjoiZGFzaGVkIn19fV1d
\[\begin{tikzcd}
	{R_\T} && X \\
	\\
	{\left(\Phi \Longrightarrow Rv_1\ldots v_n\right)_\T} && Z,
	\arrow[from=1-1, to=1-3]
	\arrow["f", from=1-3, to=3-3]
	\arrow["{f_{\Phi, R}}"', from=1-1, to=3-1]
	\arrow[from=3-1, to=3-3]
	\arrow[dashed, from=3-1, to=1-3]
\end{tikzcd}\]
there is a (not necessarily unique) diagonal filler $\left(\Phi \Longrightarrow Rv_1\ldots v_n\right)_\T \to X$, which means that $f$ has the \emph{(weak) right lifting property} with respect to $f_{\Phi, R} : R_\T \to \left(\Phi \Longrightarrow Rv_1\ldots v_n\right)_\T$.  So $f$ is convex iff for each axiom $\Phi \Longrightarrow Rv_1\ldots v_n$ of $\T \setminus \T_\Pi$ without equality, $f$ has the weak right lifting property with respect to the canonical $\Pi$-morphism $f_{\Phi, R} : R_\T \to \left(\Phi \Longrightarrow Rv_1\ldots v_n\right)_\T$.   
}
\end{rmk}

\begin{theo}
\label{convex_thm_unified}
Let $f : X \to Z$ be a convex morphism of $\T\Mod$, and let $Y$ be a $\T$-model. Then the $\Pi$-structure $P = P(Y, f)$ of \eqref{partial_product_unified} is a $\T$-model, so that $P$ is a partial product of $Y$ over $f$ in $\T\Mod$ \eqref{partial_product_prop_unified}. 
\end{theo}

\begin{proof}
We already know from Proposition \ref{partial_product_is_model} that $P$ is reflexive (i.e.~a model of $\T_\Pi$). First let $\Phi \Longrightarrow Rv_1\ldots v_n$ be an axiom of $\T\setminus \T_\Pi$ without equality, let $\kappa : \Var \to |P|$ be a valuation, and suppose that $P \models \kappa \cdot \varphi$ for each $\varphi \in \Phi$; we must show that $P \models R\kappa(v_1)\ldots\kappa(v_n)$. Let $\kappa(v_i) := (j_i, z_i) \in |P|$ for each $1 \leq i \leq n$, so that we must show $P \models R(j_1, z_1)\ldots(j_n, z_n)$. First, we must show that $Z \models Rz_1\ldots z_n$. We have the composite valuation $p \circ \kappa : \Var \to |Z|$ that satisfies $Z \models (p \circ \kappa) \cdot \varphi$ for each $\varphi \in \Phi$ because $p : P \to Z$ is a $\Pi$-morphism. Then because $Z$ is a $\T$-model, we deduce that $Z \models (p \circ \kappa) \cdot Rv_1\ldots v_n$, i.e.~that $Z \models Rz_1 \ldots z_n$. 

Now let $x_i \in f^{-1}(z_i) = f^{-1}(p(\kappa(v_i)))$ for each $1 \leq i \leq n$, suppose that $X \models Rx_1 \ldots x_n$, and let us show that $Y \models Rj_1(x_1)\ldots j_n(x_n)$. Since $Z \models (p \circ \kappa) \cdot \varphi$ for each $\varphi \in \Phi$ and $f$ is convex, there is some valuation $\kappa_X : \Var \to |X|$ such that $\kappa_X(v_i) = x_i$ for each $1 \leq i \leq n$ and $\kappa_X(v) \in f^{-1}(p(\kappa(v)))$ for each $v \in \Var(\Phi)\setminus\{v_1, \ldots, v_n\}$ and $X \models \kappa_X \cdot \varphi$ for each $\varphi \in \Phi$. We then obtain a valuation $\kappa' : \Var \to |P \times_Z X|$ given by $\kappa'(v) := (\kappa(v), \kappa_X(v))$ for each $v \in \Var(\Phi) \cup \{v_1, \ldots, v_n\}$. For each $\varphi \in \Phi$, we then readily deduce from $P \models \kappa \cdot \varphi$ and $X \models \kappa_X \cdot \varphi$ that $P \times_Z X \models \kappa' \cdot \varphi$. Then since $\varepsilon : P \times_Z X \to Y$ is a $\Pi$-morphism, we obtain $Y \models (\varepsilon \circ \kappa') \cdot \varphi$ for each $\varphi \in \Phi$. Because $Y$ is a $\T$-model, we then deduce that $Y \models (\varepsilon \circ \kappa') \cdot Rv_1\ldots v_n$, which means precisely that $Y \models Rj_1(x_1)\ldots j_n(x_n)$, as desired.

Now let $\Phi \Longrightarrow x = y$ be an axiom of $\T\setminus \T_\Pi$ with equality, let $\kappa : \Var \to |P|$ be a valuation, and suppose that $P \models \kappa \cdot \varphi$ for each $\varphi \in \Phi$; we must show that $\kappa(x) = \kappa(y) \in |P|$. Let $\kappa(x) := (j_1, z_1)$ and $\kappa(y) := (j_2, z_2)$, so that we must show $(j_1, z_1) = (j_2, z_2)$. We have the composite valuation $p \circ \kappa : \Var \to |Z|$ with $Z \models (p \circ \kappa) \cdot \varphi$ for each $\varphi \in \Phi$, since $p : P \to Z$ is a $\Pi$-morphism. Then because $Z$ is a $\T$-model, we deduce that $p(\kappa(x)) = p(\kappa(y))$, i.e.~that $z_1 = z_2 = z$. We must now show that $j_1 = j_2 : X_{f, z} \to Y$. So let $a \in f^{-1}(z)$. We have a valuation $\kappa_a : \Var \to |Y|$ given by $\kappa_a(x) := j_1(a)$ and $\kappa_a(y) := j_2(a)$. Since $\Var(\Phi) = \{x, y\}$ \eqref{Horn_theory}, for each $\varphi \in \Phi$ it readily follows from $P \models \kappa \cdot \varphi$ and $a \in f^{-1}(z) = f^{-1}(z_1) = f^{-1}(z_2)$ and the reflexivity of $\T$ that $Y \models \kappa_a \cdot \varphi$. Since $Y$ is a $\T$-model, we then deduce that $\kappa_a(x) = \kappa_a(y)$, i.e.~that $j_1(a) = j_2(a)$, as desired. This proves that $P$ is a model of $\T$.         
\end{proof}

\noindent From \ref{partial_prod_para} and Theorem \ref{convex_thm_unified} we immediately deduce the following:

\begin{theo}
\label{convex_exp_unified}
Convex morphisms of $\T\Mod$ are exponentiable. \qed
\end{theo}

\begin{rmk}
\label{nec_rmk}
{\em
For certain examples of reflexive relational Horn theories $\T$, it is known that convexity of a morphism of $\T\Mod$ is not only sufficient but also \emph{necessary} for its exponentiability. For example, when $\T$ is the reflexive relational Horn theory for preordered (partially ordered) sets (see Example \ref{without_equality_examples}.\ref{preord}), then (in view of Example \ref{convex_egg}) it is known that a morphism of $\T\Mod = \Preord$ ($\T\Mod = \Pos$) is convex iff it is exponentiable; see \cite[Theorem 2.2]{Exp_morphisms} and \cite[\S 4.1]{Exponentiation_in_V_categories}. Despite our efforts, we do not know if convexity is also necessary for exponentiability in general; the proofs of necessity in the special cases of $\Preord$ and $\Pos$ do not readily generalize to $\T\Mod$ for an arbitrary reflexive relational Horn theory $\T$.
}
\end{rmk}

\noindent Morphisms of $\T\Mod$ are \emph{always} convex with respect to axioms of $\T$ of a certain special form, as we show next.

\begin{defn}
\label{entails}
{\em
Let $\bbS$ be a relational Horn theory over a relational signature $\Sigma$. We say that $\bbS$ \textbf{entails}\footnote{It is also possible to express entailment in terms of a syntactic deducibility relation based on the axioms of $\bbS$ and certain inference rules for relational Horn formulas (cf.~\cite[Page 5]{Monadsrelational}).} a relational Horn formula $\Phi \Longrightarrow \psi$ over $\Sigma$ if it is satisfied by every $\bbS$-model. 
}
\end{defn}

\begin{defn}
\label{safe_axiom}
{\em
An axiom $\Phi \Longrightarrow Rv_1\ldots v_n$ without equality of $\T\setminus\T_\Pi$ is \textbf{safe} if there is some function $\kappa : \Var \to \{v_1, \ldots, v_n\}$ that fixes $\{v_1, \ldots, v_n\}$ (i.e.~$\kappa(v_i) = v_i$ for each $1 \leq i \leq n$) such that $\T$ entails the relational Horn formula $R v_1 \ldots v_n \Longrightarrow \kappa \cdot \varphi$ for each $\varphi \in \Phi$. The axiom $\Phi \Longrightarrow Rv_1\ldots v_n$ is \textbf{very safe} if it is safe and moreover $\Var(\Phi) \subseteq \{v_1, \ldots, v_n\}$ (so that $\T$ entails the relational Horn formula $Rv_1 \ldots v_n \Longrightarrow \varphi$ for each $\varphi \in \Phi$). 
}
\end{defn}

\begin{egg}
\label{safe_egg}
{\em
Let $\T$ be the reflexive relational Horn theory for preordered (partially ordered) sets (see Example \ref{without_equality_examples}.\ref{preord}). The only axiom of $\T\setminus\T_\Pi$ without equality is the transitivity axiom $x \leq y, y \leq z \Longrightarrow x \leq z$. This axiom is safe, because if we define $\kappa : \Var \to \{x, z\}$ by $\kappa(x) := x$ and $\kappa(y) := x$ and $\kappa(z) := z$ (and arbitrarily otherwise), then $\kappa$ fixes $\{x, z\}$ and $\T$ entails the relational Horn formulas $x \leq z \Longrightarrow x \leq x$ and $x \leq z \Longrightarrow x \leq z$.

For any reflexive relational Horn theory $\T$ and binary relation symbol $R \in \Pi$, the symmetry axiom $Rxy \Longrightarrow Ryx$ is (evidently) very safe. 
}
\end{egg}

\begin{prop}
\label{safe_prop}
Let $f : X \to Z$ be a morphism of $\T\Mod$. Then $f$ is convex with respect to all very safe axioms of $\T\setminus\T_\Pi$. 
\end{prop}

\begin{proof}
Let $\Phi \Longrightarrow Rv_1\ldots v_n$ be a very safe axiom of $\T\setminus\T_\Pi$, so that (in particular) $\Var(\Phi) \subseteq \{v_1, \ldots, v_n\}$. Let $\kappa_Z : \Var \to |Z|$ be a valuation such that $Z \models \kappa_Z \cdot \varphi$ for each $\varphi \in \Phi$, let $x_i \in f^{-1}(\kappa_Z(v_i))$ for each $1 \leq i \leq n$, and suppose that $X \models Rx_1 \ldots x_n$. We then have a valuation $\kappa : \Var \to |X|$ given by $\kappa(v_i) := x_i$ for each $1 \leq i \leq n$ such that $X \models \kappa \cdot \varphi$ for each $\varphi \in \Phi$, because $X \models R\kappa(v_1)\ldots\kappa(v_n)$ and $\T$ entails $Rv_1\ldots v_n \Longrightarrow \varphi$ for each $\varphi \in \Phi$.    
\end{proof}

We shall now specialize the preceding definitions and results to provide useful sufficient conditions for the exponentiability of \emph{objects} of $\T\Mod$.

\begin{para}
\label{partial_product_object}
{\em
For a $\T$-model $X$ and the unique morphism $!_X : X \to 1$, we now simplify the construction of the $\Pi$-structure $P = P(Y, !_X) = Y^X$ of Definition \ref{partial_product_unified} for a $\T$-model $Y$. We have $\left|Y^X\right| = \Str(\Pi)(X, Y)$, the set of $\Pi$-morphisms $X \to Y$. For each $R \in \Pi$ of arity $n \geq 1$ and any $\Pi$-morphisms $h_1, \ldots, h_n : X \to Y$, we have $Y^X \models Rh_1 \ldots h_n$ iff $X \models Rx_1 \ldots x_n$ implies $Y \models Rh_1(x_1)\ldots h_n(x_n)$ for all $x_1, \ldots, x_n \in |X|$. The evaluation $\Pi$-morphism $\varepsilon : Y^X \times X \to Y$ is given by $\varepsilon(h, x) := h(x)$ for $h \in \Str(\Pi)(X, Y)$ and $x \in |X|$.  
}
\end{para} 

\noindent Definition \ref{convex_unified} now specializes as follows:

\begin{defn}
\label{convex_object}
{\em
Let $X$ be a $\T$-model, and let $\Phi \Longrightarrow Rv_1\ldots v_n$ be an axiom of $\T\setminus\T_\Pi$ without equality. We say that $X$ is \textbf{convex with respect to (the axiom) $\Phi \Longrightarrow Rv_1 \ldots v_n$} if the unique morphism $!_X : X \to 1$ is convex with respect to $\Phi \Longrightarrow Rv_1 \ldots v_n$, i.e.~if for all $x_1, \ldots, x_n \in |X|$ such that $X \models Rx_1 \ldots x_n$, there is a valuation $\kappa : \Var \to |X|$ such that $\kappa(v_i) = x_i$ for each $1 \leq i \leq n$ and $X \models \kappa \cdot \varphi$ for each $\varphi \in \Phi$. We say that $X$ is \textbf{convex} if it is convex with respect to each axiom of $\T\setminus\T_\Pi$ without equality.
}  
\end{defn}

\noindent Theorems \ref{convex_thm_unified} and \ref{convex_exp_unified} now specialize to yield the following:

\begin{theo}
\label{convex_exp_object}
Convex $\T$-models are exponentiable. \qed
\end{theo}

\begin{prop}
\label{safe_object_prop}
Let $X$ be a $\T$-model. Then $X$ is convex with respect to all safe axioms of $\T\setminus\T_\Pi$. 
\end{prop}

\begin{proof}
Let $\Phi \Longrightarrow Rv_1 \ldots v_n$ be a safe axiom of $\T\setminus\T_\Pi$. Then there is some function $\kappa : \Var \to \{v_1, \ldots, v_n\}$ that fixes $\{v_1, \ldots, v_n\}$ and is such that $\T$ entails the relational Horn formula $Rv_1 \ldots v_n \Longrightarrow \kappa \cdot \varphi$ for each $\varphi \in \Phi$. Let $x_1, \ldots, x_n \in |X|$ and suppose that $X \models Rx_1 \ldots x_n$. Let $\iota : \{v_1, \ldots, v_n\} \to |X|$ be the function defined by $\iota(v_i) := x_i$ for each $1 \leq i \leq n$. We then have a valuation $\kappa_X := \iota \circ \kappa: \Var \to |X|$ satisfying $\kappa_X(v_i) = x_i$ for each $1 \leq i \leq n$ such that $X \models \kappa_X \cdot \varphi$ for each $\varphi \in \Phi$, because $X \models R\kappa_X(v_1)\ldots \kappa_X(v_n)$ and $\T$ entails $Rv_1 \ldots v_n \Longrightarrow \kappa \cdot \varphi$.    
\end{proof}

\noindent From Theorem \ref{convex_exp_unified}, Proposition \ref{safe_prop}, Theorem \ref{convex_exp_object}, and Proposition \ref{safe_object_prop}, we now immediately deduce the following result:

\begin{theo}
\label{lcc_thm}
Let $\T$ be a reflexive relational Horn theory over a relational signature $\Pi$, and suppose that all axioms of $\T\setminus\T_\Pi$ without equality are safe (resp.~very safe). Then $\T\Mod$ is cartesian closed (resp.~locally cartesian closed). \qed
\end{theo}

\begin{egg}
\label{preord_cc}
{\em
We saw in Example \ref{safe_egg} that if $\T$ is the reflexive relational Horn theory for preordered or partially ordered sets, then all axioms of $\T\setminus\T_\Pi$ without equality are safe. So from Theorem \ref{lcc_thm} we recover the well-known facts that $\Preord$ and $\Pos$ are cartesian closed categories. We also saw in Example \ref{safe_egg} that if $\T$ is any reflexive relational Horn theory, then for each binary relation symbol $R \in \Pi$, the symmetry axiom $Rxy \Longrightarrow Ryx$ is very safe. So if $\Pi$ consists of a single binary relation symbol $R$ and $\T$ is the relational Horn theory over $\Pi$ consisting of the two axioms $\Longrightarrow Rxx$ and $Rxy \Longrightarrow Ryx$, then from Theorem \ref{lcc_thm} we also recover the well-known fact that $\T\Mod$, which is the category of sets equipped with a (binary) reflexive and symmetric relation, is locally cartesian closed. 
}
\end{egg}

A \emph{topological universe} \cite[Definition 28.21]{AHS} is a well-fibred topological category over $\Set$ that is also a quasitopos. It is remarked in \cite[Example 28.23]{AHS} that (using our notation) when $\Pi$ is the discrete signature consisting of a single binary relation symbol, the category $\T_\Pi\Mod$ (whose objects are sets equipped with a (binary) reflexive relation) is a topological universe. We now show that this result holds more generally. 

\begin{theo}
\label{quasitopos_thm}
Let $\T$ be reflexive relational Horn theory without equality, and suppose that all axioms of $\T\setminus\T_\Pi$ are very safe. Then $\T\Mod$ is a topological universe. 
\end{theo}

\begin{proof}
We know from Theorem \ref{lcc_thm} that $\T\Mod$ is locally cartesian closed. Since $\T\Mod$ is topological over $\Set$ \eqref{relational_top} and well-fibred \eqref{refl_para}, we now deduce from \ref{quasitopos_para} that $\T\Mod$ is a quasitopos and hence a topological universe. 
\end{proof}

A \emph{topology} $\mathcal{P}$ (see \cite[V.4.1]{Monoidaltop}) on a finitely complete category $\C$ is a class of morphisms of $\C$ that contains all isomorphisms, is closed under composition, and is stable under pullback. It is known (see e.g.~\cite[V.4.D]{Monoidaltop} or \cite[Corollaries 1.3 and 1.4]{Cartesianness}) that the exponentiable morphisms of $\C$ form a topology $\mathsf{Exp}(\C)$ on $\C$.

\begin{defn}
{\em
We let $\Convex(\T\Mod)$ be the class of convex morphisms of $\T\Mod$.
}
\end{defn} 

\noindent Since we in general only know that $\Convex(\T\Mod) \subseteq \mathsf{Exp}(\T\Mod)$ (see Theorem \ref{convex_exp_unified} and Remark \ref{nec_rmk}), we cannot directly deduce that $\Convex(\T\Mod)$ is a topology on $\T\Mod$ from the fact that $\mathsf{Exp}(\T\Mod)$ is a topology on $\T\Mod$; but the former claim is nevertheless true:  

\begin{prop}
\label{convex_iso}
$\Convex(\T\Mod)$ is a topology on $\T\Mod$.  
\end{prop}

\begin{proof}
From Remark \ref{lifting_rmk} we know that a morphism of $\T\Mod$ is convex iff it has the weak right lifting property with respect to a certain set of morphisms of $\T\Mod$, and it is well known (and straightforward to prove) that this entails that $\Convex(\T\Mod)$ is a topology.      
\end{proof}

We conclude this section with a useful alternative sufficient condition for $\T\Mod$ to be cartesian closed (cf.~Theorem \ref{lcc_thm}). An object $Y$ of a category $\C$ with finite products is sometimes said to be \emph{exponentiating} if an exponential $\varepsilon_X : Y^X \times X \to Y$ (i.e.~a coreflection of $Y$ along $X \times (-) : \C \to \C$) exists for each $X \in \ob\C$. If $\Pi$ contains only binary relation symbols, then we say that a $\Pi$-structure $X$ is \emph{transitive} if the binary relation $R^X$ on $|X|$ is transitive for each $R \in \Pi$, and we say that a relational Horn theory $\T$ over $\Pi$ is \emph{transitive} if every $\T$-model is transitive. 

\begin{theo}
\label{cart_closed_thm}
Suppose that $\Pi$ contains only binary relation symbols, and let $\T$ be a reflexive relational Horn theory over $\Pi$. Then every transitive $\T$-model is exponentiating. So if $\T$ is a reflexive and transitive relational Horn theory over $\Pi$, then $\T\Mod$ is cartesian closed.
\end{theo}

\begin{proof}
Let $Y$ be a transitive $\T$-model and let $X$ be a $\T$-model. We will show that the $\Pi$-structure $Y^X$ of \ref{partial_product_object} is a $\T$-model, so that the desired result will then follow by Proposition~\ref{partial_product_prop_unified}. We first claim that for each $R \in \Pi$ and all $h_1, h_2 \in \left|Y^X\right| = \Str(\Pi)(X, Y)$, we have $Y^X \models Rh_1h_2$ (as in \ref{partial_product_object}) iff $Y \models Rh_1(x)h_2(x)$ for all $x \in |X|$. Suppose first that $Y^X \models Rh_1h_2$ as in \ref{partial_product_object}, and let $x \in |X|$. Then since $X \models Rxx$, we deduce that $Y \models Rh_1(x)h_2(x)$, as desired. Conversely, suppose that $Y \models Rh_1(x)h_2(x)$ for each $x \in |X|$, and suppose that $X \models Rx_1x_2$; we must show that $Y \models Rh_1(x_1)h_2(x_2)$. In particular we have $Y \models Rh_1(x_1)h_2(x_1)$. Since $h_2 : X \to Y$ is a $\Pi$-morphism, we also have $Y \models Rh_2(x_1)h_2(x_2)$. Then from transitivity of $Y$ we obtain $Y \models Rh_1(x_1)h_2(x_2)$, as desired. 

We now show that $Y^X$ is a $\T$-model. For any $\Pi$-edge $(R, (h_1, h_2))$ in $\left|Y^X\right| = \Str(\Pi)(X, Y)$ and any $x \in |X|$, we write $(R, (h_1, h_2))(x)$ for the $\Pi$-edge $(R, (h_1(x), h_2(x)))$ in $|Y|$. Let $\Phi \Longrightarrow \psi$ be an axiom of $\T\setminus\T_\Pi$, and let $\kappa : \Var \to \left|Y^X\right| = \Str(\Pi)(X, Y)$ be a valuation such that $\kappa \cdot \varphi \in \E\left(Y^X\right)$ for each $\varphi \in \Phi$. Supposing first that $\Phi \Longrightarrow \psi$ is without equality, we must show that $\kappa \cdot \psi \in \E\left(Y^X\right)$. By the previous paragraph, this means showing for each $x \in |X|$ that $(\kappa \cdot \psi)(x) \in \E(Y)$. Given $x \in |X|$, we have a valuation $\kappa_x : \Var \to |Y|$ defined by $\kappa_x(v) := \kappa(v)(x)$ for each $v \in \Var$. For each $\varphi \in \Phi$ we have $\kappa_x \cdot \varphi = (\kappa \cdot \varphi)(x) \in \E(Y)$ by hypothesis (and the previous paragraph), so that $(\kappa \cdot \psi)(x) = \kappa_x \cdot \psi \in \E(Y)$ because $Y$ is a $\T$-model, as desired. Now suppose that $\psi \equiv v_1 = v_2$, and let us show that $\kappa(v_1) = \kappa(v_2)$, i.e.~that $\kappa(v_1)(x) = \kappa(v_2)(x)$ for all $x \in |X|$. As above, we have the valuation $\kappa_x : \Var \to |Y|$ with $\kappa_x(v) = \kappa(v)(x)$ for each $v \in \Var$, and for each $\varphi \in \Phi$ we have $Y \models \kappa_x \cdot \varphi$, so that $\kappa_x(v_1) = \kappa_x(v_2)$ because $Y$ is a $\T$-model, i.e. $\kappa(v_1)(x) = \kappa(v_2)(x)$, as desired.   
\end{proof}

\begin{egg}
\label{trans_egg}
{\em
In Example \ref{preord_cc} we noted that one explanation for the cartesian closure of $\Preord$ and $\Pos$ is the fact that the axioms without equality of the corresponding relational Horn theories are all safe. Since these relational Horn theories are reflexive and transitive (and their relational signature contains only binary relation symbols), Theorem \ref{cart_closed_thm} also provides another explanation for the cartesian closure of $\Preord$ and $\Pos$.
}
\end{egg}

\begin{rmk}
\label{cart_closed_rmk}
{\em
The condition that $\T$ be (reflexive and) transitive is sufficient (by Theorem \ref{cart_closed_thm}) but certainly not \emph{necessary} for $\T\Mod$ to be cartesian closed. For example, if $\Pi$ contains a single binary relation symbol $R$, then $\T_\Pi$ consists of the single axiom $\Longrightarrow Rxx$, and $\T_\Pi\Mod$ is (locally) cartesian closed by Theorem \ref{lcc_thm}, even though there clearly exist $\T_\Pi$-models that are not transitive (i.e.~sets equipped with a binary reflexive relation that is not transitive).
}
\end{rmk}

\section{Convexity in the non-discrete case}
\label{non_disc_section}

In \S\ref{disc_section} we considered reflexive relational Horn theories $\T$ over discrete preordered relational signatures $\Pi$, and we provided sufficient conditions for objects and morphisms of $\T\Mod$ to be exponentiable. In this final section we shall consider the case where the preordered relational signature $\Pi$ is not discrete. In fact, we shall suppose throughout this section that the preordered relational signature $\Pi$ is a \emph{complete Heyting algebra} \eqref{preord_sig}. For example, if $(\V, \leq, \tensor, \sfk)$ is a commutative unital quantale such that $(\V, \leq)$ is a complete Heyting algebra, then the preordered relational signature $\Pi_\V$ of Example \ref{without_equality_examples}.\ref{quantale} is a complete Heyting algebra (see Example \ref{preord_sig_egg}). Note that if $\sfk = \top$ (the top element of $\V$), then we may identify the relational Horn theory $\T_{\Pi_\V}$ of Definition \ref{T_Pi} with the relational Horn theory $\T_{\V\RGph}$ of Example \ref{without_equality_examples}.\ref{quantale}\footnote{Technically $\T_{\Pi_\V}$ contains the reflexivity axioms $\Longrightarrow x \sim_v x$ for all $v \in \V$, whereas the only reflexivity axiom that $\T_{\V\RGph}$ contains is $\Longrightarrow x \sim_\sfk x$; but since $\sfk = \top$, it follows that $\T_{\V\RGph}$ entails the relational Horn formulas $\Longrightarrow x \sim_v x$ for all $v \in \V$, so we can assume w.l.o.g.~that they are axioms of $\T_{\V\RGph}$.}.   

We shall identify a useful sufficient condition for a morphism of $\T\Mod$ to be exponentiable, where $\T$ is a \emph{schematic extension} of $\T_\Pi$ (see Definition \ref{schematic_theory}). As a result, we shall recover a known sufficient condition for a morphism of $\V\Cat$ to be exponentiable, where $(\V, \leq, \tensor, \sfk)$ is a commutative unital quantale such that $(\V, \leq)$ is a complete Heyting algebra and $\sfk = \top$ (see \cite[Theorem 3.4]{Exponentiation_in_V_categories}, where it is shown that this condition is also \emph{necessary} for exponentiability). In fact, our approach in this section is greatly influenced by the characterization of exponentiable objects and morphisms in $\V\Cat \cong \T_{\V\Cat}\Mod$ (see Example \ref{without_equality_examples}.\ref{quantale}), as the reader can hopefully glean from Examples \ref{schema_egg}, \ref{schema_convex_egg}, and \ref{schematic_egg} below.

\begin{defn}
\label{axiom_schema}
{\em
Let $n \geq 1$, and let $\calS$ be a relation symbol of arity $n$ that is not in $\Pi$. An \textbf{$n$-ary axiom schema over $\Pi$} is a triple $(\Phi, \psi, \sigma)$ consisting of a set $\Phi$ of $\{\calS\}$-edges in $\Var$, a $\{\calS\}$-edge $\psi$ in $\Var$, and a function $\sigma : \Pi(n)^\Phi \to \Pi(n)$. Given a $\{\calS\}$-edge $\varphi \equiv \calS v_1 \ldots v_n$ in $\Var$ and an element $R \in \Pi(n)$, we let $\varphi_R$ be the $\Pi$-edge $Rv_1 \ldots v_n$ in $\Var$. Given a set $\Phi$ of $\{\calS\}$-edges in $\Var$ and a tuple $\overline{R} = \left(R_\varphi\right)_{\varphi \in \Phi} \in \Pi(n)^\Phi$, we define the set of $\Pi$-edges $\Phi_{\Rbar} := \{\varphi_{R_\varphi} \mid \varphi \in \Phi\}$ in $\Var$. An \textbf{instance} of an $n$-ary axiom schema $(\Phi, \psi, \sigma)$ over $\Pi$ is a pair $\left(\Rbar \in \Pi(n)^\Phi, \Phi_{\Rbar} \Longrightarrow \psi_{\sigma\left(\Rbar\right)}\right)$. An \textbf{axiom schema over $\Pi$} is an $n$-ary axiom schema over $\Pi$ for some $n \geq 1$.      
}
\end{defn}

\begin{egg}
\label{schema_egg}
{\em
Let $(\V, \leq, \tensor, \top)$ be a commutative unital quantale such that $(\V, \leq)$ is a complete Heyting algebra.
The \emph{generalized transitivity} axiom schema is the binary axiom schema $\left(\{x \calS y, y \calS z\}, x \calS z, \sigma\right)$ over $\Pi_\V$, where $\sigma : \Pi_\V(2) \times \Pi_\V(2) \to \Pi_\V(2)$ is given by 

\noindent $\sigma(\sim_v, \sim_{v'}) := \ \sim_{v \tensor v'}$ for $v, v' \in \V$. An instance of the generalized transitivity axiom schema is thus (or may be identified with) a pair $\left((v, v') \in \V^2, \{x \sim_v y, y \sim_{v'} z\} \Longrightarrow x \sim_{v \tensor v'} z\right)$. 

The \emph{symmetry} axiom schema is the binary axiom schema $\left(\{x \calS y\}, y \calS x, 1_{\Pi_\V(2)}\right)$ over $\Pi_\V$, an instance of which is a pair $\left(v \in \V, x \sim_v y \Longrightarrow y \sim_v x\right)$. 
}
\end{egg}

\noindent For any set $\Phi$ and $n \geq 1$ and $\Rbar, \Sbar \in \Pi(n)^\Phi$, we define $\Rbar \wedge \Sbar := \left(R_\varphi \wedge S_\varphi\right)_{\varphi \in \Phi} \in \Pi(n)^\Phi$.

\begin{defn}
\label{schema_convex}
{\em
Let $f : X \to Z$ be a morphism of $\T_\Pi$-models, let $(\Phi, \calS v_1 \ldots v_n, \sigma)$ be an $n$-ary axiom schema over $\Pi$, and let $\left(\Rbar \in \Pi(n)^\Phi, \Phi_{\Rbar} \Longrightarrow \sigma\left(\Rbar\right)v_1 \ldots v_n\right)$ be an instance. We say that $f$ is \textbf{convex with respect to the instance $\left(\Rbar, \Phi_{\Rbar} \Longrightarrow \sigma\left(\Rbar\right)v_1 \ldots v_n\right)$ of the axiom schema $(\Phi, \calS v_1 \ldots v_n, \sigma)$} if $f$ satisfies the following condition:

Let $\kappa_Z : \Var \to |Z|$ be a valuation satisfying $Z \models \kappa_Z \cdot \varphi_{R_\varphi}$ for each $\varphi \in \Phi$, and let $x_i \in f^{-1}(\kappa_Z(v_i))$ for each $1 \leq i \leq n$. We say that a valuation $\kappa : \Var \to |X|$ is \textbf{good} if $\kappa(v_i) = x_i$ for each $1 \leq i \leq n$ and $\kappa(v) \in f^{-1}(\kappa_Z(v))$ for all $v \in \Var(\Phi)\setminus\{v_1, \ldots, v_n\}$. For each good valuation $\kappa : \Var \to |X|$, we define
\[ R_\kappa := \bigvee\left\{\sigma\left(\Rbar \wedge \Sbar\right) \mid \Sbar \in \Pi(n)^\Phi \text{ and } X \models \kappa \cdot \varphi_{S_\varphi} \ \forall \varphi \in \Phi\right\}. \]
Then for each $T \in \Pi(n)$ such that $T \leq \sigma\left(\Rbar\right)$ and $X \models Tx_1\ldots x_n$, we require that
\[ T \leq \bigvee\left\{ R_{\kappa} \mid \kappa : \Var \to |X| \text{ is good} \right\}. \] We say that $f$ is \textbf{convex with respect to an axiom schema} if $f$ is convex with respect to each instance of the axiom schema.  
}
\end{defn}

\begin{egg}
\label{schema_convex_egg}
{\em
Let $(\V, \leq, \tensor, \top)$ be a commutative unital quantale such that $(\V, \leq)$ is a complete Heyting algebra. Let $f : X \to Z$ be a morphism of $\T_{\Pi_\V}$-models (i.e.~$\T_{\V\RGph}$-models), so that $f : (|X|, d_X) \to (|Z|, d_Z)$ is a $\V$-functor between the corresponding reflexive $\V$-graphs; see Example \ref{without_equality_examples}.\ref{quantale}. For any $\T_{\Pi_\V}$-model (i.e.~$\T_{\V\RGph}$-model) $Y$ and $y, y' \in |Y|$, we have $d_Y(y, y') = \bigvee\{u \in \V \mid Y \models y \sim_u y'\}$ (see \cite[Appendix]{Extensivity}), so for any $v \in \V$ it readily follows that $v \leq d_Y(y, y')$ iff $Y \models y \sim_v y'$. Consider the generalized transitivity axiom schema $\left(\{x \calS y, y \calS z\}, x \calS z, \sigma\right)$ over $\Pi_\V$ from Example \ref{schema_egg}.  

We now verify that $f : X \to Z$ is convex with respect to each instance of the generalized transitivity axiom schema iff $f : (|X|, d_X) \to (|Z|, d_Z)$ satisfies the condition of \cite[Theorem 3.4]{Exponentiation_in_V_categories}, namely that for all $x_1, x_3 \in |X|$, $z_2 \in |Z|$, and $v, v' \in \V$ with $v \leq d_Z(f(x_1), z_2)$ and $v' \leq d_Z(z_2, f(x_2))$, we have
\[ d_X(x_1, x_3) \wedge (v \tensor v') \leq \bigvee_{x_2 \in f^{-1}(z_2)} \left(d_X(x_1, x_2) \wedge v\right) \tensor \left(d_X(x_2, x_3) \wedge v'\right). \] 
Suppose first that $f$ is convex with respect to each instance of the generalized transitivity axiom schema, and let $x_1, x_3 \in |X|$, $z_2 \in |Z|$, and $v, v' \in \V$ be as in the condition. Then we have $Z \models f(x_1) \sim_v z_2$ and $Z \models z_2 \sim_{v'} f(x_2)$, and since $d_X(x_1, x_3) \wedge (v \tensor v') \leq v \tensor v'$ and $X \models x_1 \sim_{d_X(x_1, x_3) \wedge (v \tensor v')} x_3$, we deduce from the hypothesis on $f$ that 
\[ d_X(x_1, x_3) \wedge (v \tensor v') \] \[ \leq \bigvee_{x_2 \in f^{-1}(z_2)}\left\{(u \wedge v) \tensor (u' \wedge v') \mid u, u' \in \V \text{ and } X \models x_1 \sim_u x_2 \text{ and } X \models x_2 \sim_{u'} x_3\right\}. \] Now for each $x_2 \in f^{-1}(z_2)$ we have
\begin{align*}
&\ \ \ \ \left(d_X(x_1, x_2) \wedge v\right) \tensor \left(d_X(x_2, x_3) \wedge v'\right) \\	
&= \left(\bigvee\{u \in \V \mid X \models x_1 \sim_u x_2\} \wedge v\right) \tensor \left(\bigvee\{u' \in \V \mid X \models x_2 \sim_{u'} x_3\} \wedge v'\right) \\
&= \bigvee\left\{u \wedge v \mid u \in \V \text{ and } X \models x_1 \sim_u x_2\right\} \tensor \bigvee\left\{u' \wedge v' \mid u' \in \V \text{ and } X \models x_2 \sim_{u'} x_3\right\} \\
&= \left\{(u \wedge v) \tensor (u' \wedge v') \mid u, u' \in \V \text{ and } X \models x_1 \sim_u x_2 \text{ and } X \models x_2 \sim_{u'} x_3\right\}, 
\end{align*}
where the second equality holds because $\V$ is a Heyting algebra and the last because $\V$ is a quantale. Thus, $f$ satisfies the condition of \cite[Theorem 3.4]{Exponentiation_in_V_categories}. Conversely, suppose that $f$ satisfies this condition, let $v, v' \in \V$, and let $x_1, x_3 \in |X|$ and $z_2 \in |Z|$ be such that $Z \models f(x_1) \sim_v z_2$ and $Z \models z_2 \sim_{v'} f(x_3)$. For each $w \leq v \tensor v'$ such that $X \models x_1 \sim_w x_3$, we must show that 
\[ w \leq \bigvee_{x_2 \in f^{-1}(z_2)}\left\{(u \wedge v) \tensor (u' \wedge v') \mid u, u' \in \V \text{ and } X \models x_1 \sim_u x_2 \text{ and } X \models x_2 \sim_{u'} x_3\right\}, \] i.e.~that $w \leq \bigvee_{x_2 \in f^{-1}(z_2)}\left(d_X(x_1, x_2) \wedge v\right) \tensor \left(d_X(x_2, x_3) \wedge v'\right)$ (by the calculation above). By hypothesis on $Z$ we have $v \leq d_Z(f(x_1), z_2)$ and $v' \leq d_Z(z_2, f(x_3))$, so we deduce from the condition on $f$ that
\[ d_X(x_1, x_3) \wedge (v \tensor v') \leq \bigvee_{x_2 \in f^{-1}(z_2)} \left(d_X(x_1, x_2) \wedge v\right) \tensor \left(d_X(x_2, x_3) \wedge v'\right), \] which yields the result because $w \leq d_X(x_1, x_3)$ (since $X \models x_1 \sim_w x_3$) and $w \leq v \tensor v'$.  
}
\end{egg}

\begin{defn}
\label{schematic_theory}
{\em
Let $\T$ be a relational Horn theory over $\Pi$. We say that $\T$ is a \textbf{schematic extension of $\T_\Pi$ given by a set $\mathscr{S} = \left\{(\Phi_s, \psi_s, \sigma_s) \mid s \in \scrS\right\}$ of axiom schemas over $\Pi$} if the axioms of $\T$ are those of $\T_\Pi$, together with all instances of all axiom schemas in $\scrS$, together with (perhaps) some axioms with equality. We say that a morphism of $\T$-models is \textbf{convex} if it is convex with respect to each axiom schema in $\scrS$.  
}
\end{defn}

\begin{egg}
\label{schematic_egg}
{\em
Let $(\V, \leq, \tensor, \top)$ be a commutative unital quantale such that $(\V, \leq)$ is a complete Heyting algebra. The relational Horn theory $\T_{\V\Cat}$ over $\Pi_\V$ of Example \ref{without_equality_examples}.\ref{quantale} is a schematic extension of $\T_{\Pi_\V} = \T_{\V\RGph}$ given by the set consisting of just the generalized transitivity axiom schema of Example \ref{schema_egg}. The relational Horn theory $\T_{\PMet_\V}$ over $\Pi_\V$ of Example \ref{without_equality_examples}.\ref{quantale} is a schematic extension of $\T_{\Pi_\V} = \T_{\V\RGph}$ given by the set consisting of the generalized transitivity axiom schema and the symmetry axiom schema of Example \ref{schema_egg}. The relational Horn theory $\T_{\Met_\V}$ over $\Pi_\V$ of Example \ref{without_equality_examples}.\ref{quantale} is a schematic extension of $\T_{\Pi_\V} = \T_{\V\RGph}$ given by the set consisting of the generalized transitivity axiom schema and the symmetry axiom schema of Example \ref{schema_egg}, together with the axiom $x \sim_\top y \Longrightarrow x = y$.   
}
\end{egg}

\begin{theo}
\label{schema_convex_part_prod}
Let $\T$ be a schematic extension of $\T_\Pi$ given by a set $\scrS$ of axiom schemas over $\Pi$, and let $f : X \to Z$ be a convex morphism of $\T$-models. Then for every $\T$-model $Y$, the $\Pi$-structure $P = P(Y, f)$ of \eqref{partial_product_unified} is a model of $\T$.  
\end{theo}

\begin{proof}
In view of Proposition \ref{partial_product_is_model}, it remains to show that $P$ is a model of each instance of each axiom schema in $\scrS$, and of each axiom of $\T$ with equality. The latter assertion is proved exactly as in the proof of Theorem \ref{convex_thm_unified}. So let $(\Phi, \calS v_1 \ldots v_n, \sigma)$ be an axiom schema in $\scrS$, and let $\left(\Rbar \in \Pi(n)^\Phi, \Phi_{\Rbar} \Longrightarrow \sigma\left(\Rbar\right)v_1\ldots v_n\right)$ be an instance. Let $\lambda : \Var \to |P|$ be a valuation such that $P \models \lambda \cdot \varphi_{R_\varphi}$ for each $\varphi \in \Phi$, and let us show that $P \models \sigma\left(\Rbar\right)\lambda(v_1)\ldots\lambda(v_n)$. Let $\lambda(v_i) := (j_i, z_i)$ for each $1 \leq i \leq n$, so that we must show $P \models \sigma\left(\Rbar\right)(j_1, z_1)\ldots(j_n, z_n)$. The proof that $Z \models \sigma\left(\Rbar\right)z_1\ldots z_n$ is exactly as in the proof of Theorem \ref{convex_thm_unified}; in particular, we have the valuation $\kappa_Z := p \circ \lambda : \Var \to |Z|$ satisfying $Z \models \kappa_Z \cdot \varphi_{R_\varphi}$ for each $\varphi \in \Phi$. Now let $x_i \in f^{-1}(z_i) = f^{-1}(\kappa_Z(v_i))$ for each $1 \leq i \leq n$, let $T \leq \sigma\left(\Rbar\right)$ be such that $X \models Tx_1\ldots x_n$, and let us show that $Y \models Tj_1(x_1)\ldots j_n(x_n)$. Because $f$ is convex and $Y$ is a $\T_\Pi$-model, this will be true if $Y \models R_\kappa j_1(x_1)\ldots j_n(x_n)$ for every good valuation $\kappa : \Var \to |X|$. Given a good valuation $\kappa : \Var \to |X|$, let $\Sbar \in \Pi(n)^\Phi$ satisfy $X \models \kappa \cdot \varphi_{S_\varphi}$ for each $\varphi \in \Phi$; since $Y$ is a $\T_\Pi$-model, it then suffices to show that $Y \models \sigma\left(\Rbar \wedge \Sbar\right) j_1(x_1)\ldots j_n(x_n)$.  

For each $\varphi \in \Phi$, it follows from $P \models \lambda \cdot \varphi_{R_\varphi}$ and the fact that $P$ is a $\T_\Pi$-model that $P \models \lambda \cdot \varphi_{(R_\varphi \wedge S_\varphi)}$. Similarly, for each $\varphi \in \Phi$ we have $X \models \kappa \cdot \varphi_{S_\varphi}$ and hence $X \models \kappa \cdot \varphi_{(R_\varphi \wedge S_\varphi)}$. We have a valuation $\kappa' : \Var \to |P \times_Z X|$ given by $\kappa'(v) := (\lambda(v), \kappa(v))$ for each $v \in \Var$ that therefore satisfies $P \times_Z X \models \kappa' \cdot \varphi_{(R_\varphi \wedge S_\varphi)}$ for each $\varphi \in \Phi$. Since $\varepsilon : P \times_Z X \to Y$ is a $\Pi$-morphism, we then deduce that $Y \models (\varepsilon \circ \kappa') \cdot \varphi_{(R_\varphi \wedge S_\varphi)}$ for each $\varphi \in \Phi$. Because $Y$ is a $\T$-model and $\left(\Rbar \wedge \Sbar \in \Pi(n)^\Phi, \Phi_{\Rbar \wedge \Sbar} \Longrightarrow \sigma\left(\Rbar \wedge \Sbar\right)v_1\ldots v_n\right)$ is an instance of the given axiom schema $(\Phi, \calS v_1 \ldots v_n, \sigma)$ of $\scrS$, it then follows that $Y \models \sigma\left(\Rbar \wedge \Sbar\right)\varepsilon(\kappa'(v_1))\ldots \varepsilon(\kappa'(v_n))$, i.e.~that $Y \models \sigma\left(\Rbar \wedge \Sbar\right) j_1(x_1)\ldots j_n(x_n)$, as desired.           
\end{proof}

\noindent From \ref{partial_prod_para}, Proposition  \ref{partial_product_prop_unified}, and Theorem \ref{schema_convex_part_prod} we immediately obtain the following theorem:

\begin{theo}
\label{schema_convex_exp}
Let $\T$ be a schematic extension of $\T_\Pi$. Then every convex morphism of $\T\Mod$ is exponentiable. \qed
\end{theo}

\begin{rmk}
\label{nec_rmk_2}
{\em
For certain examples of schematic extensions $\T$ of $\T_\Pi$, it is known that convexity of a morphism of $\T\Mod$ is not only sufficient but also \emph{necessary} for its exponentiability. For example, when $\T = \T_{\V\Cat}$ (see Example \ref{schematic_egg}) for a commutative unital quantale $(\V, \leq, \tensor, \top)$ such that $(\V, \leq)$ is a complete Heyting algebra, then (in view of Example \ref{schema_convex_egg}) it is known that a morphism of $\T_{\V\Cat}\Mod \cong \V\Cat$ is convex iff it is exponentiable; see \cite[Theorem 3.4]{Exponentiation_in_V_categories}. As in Remark \ref{nec_rmk}, we do not know if convexity is also necessary for exponentiability in general; the proof of necessity given in \cite[Theorem 3.4]{Exponentiation_in_V_categories} for $\T = \T_{\V\Cat}$ does not readily generalize to $\T\Mod$ for an arbitrary schematic extension $\T$ of $\T_\Pi$.
}
\end{rmk}

\begin{defn}
\label{schema_safe}
{\em
Let $\T$ be a schematic extension of $\T_\Pi$ given by a set $\scrS$ of axiom schemas over $\Pi$, and let $(\Phi, \calS v_1 \ldots v_n, \sigma)$ be an axiom schema in $\scrS$. This axiom schema is \textbf{safe} if $\sigma\left(\Rbar \wedge S\right) = \sigma\left(\Rbar\right) \wedge S$ for all $\Rbar \in \Pi(n)^\Phi$ and $S \in \Pi(n)$ (where $\Rbar \wedge S := \left(R_\varphi \wedge S \right)_{\varphi \in \Phi}$) and there is some function $\kappa : \Var \to \{v_1, \ldots, v_n\}$ that fixes $\{v_1, \ldots, v_n\}$ such that $\T$ entails $\sigma\left(\Rbar\right)v_1\ldots v_n \Longrightarrow \kappa \cdot \varphi_{R_\varphi}$ for all $\Rbar \in \Pi(n)^\Phi$ and $\varphi \in \Phi$. The axiom schema is \textbf{very safe} if it is safe and moreover $\Var(\Phi) \subseteq \{v_1, \ldots, v_n\}$ (so that $\T$ entails $\sigma\left(\Rbar\right)v_1\ldots v_n  \Longrightarrow \varphi_{R_\varphi}$ for all $\Rbar \in \Pi(n)^\Phi$ and $\varphi \in \Phi$).
}
\end{defn}

\begin{egg}
\label{schema_safe_egg}
{\em
Let $(\V, \leq, \tensor, \top)$ be a commutative unital quantale such that $(\V, \leq)$ is a complete Heyting algebra. 
Consider the relational Horn theory $\T_{\V\Cat}$ over $\Pi_\V$ of Example \ref{without_equality_examples}.\ref{quantale}, which is a schematic extension of $\T_{\Pi_\V} = \T_{\V\RGph}$ given by the set consisting of just the generalized transitivity axiom schema of Example \ref{schema_egg}. This axiom schema is \emph{not} safe in general, for otherwise $\T_{\V\Cat}\Mod \cong \V\Cat$ would be cartesian closed by Theorem \ref{lcc_thm_V} below, which is not true in general. However, the generalized transitivity axiom schema \emph{is} safe when $(\V, \leq)$ is totally ordered and $\tensor = \wedge$ is the binary meet operation of $(\V, \leq)$: for we of course have $(v_1 \wedge v_2) \wedge v = (v_1 \wedge v) \wedge (v_2 \wedge v)$ for all $v, v_1, v_2 \in \V$; and for any $v, v' \in \V$, we have w.l.o.g.~$v \leq v'$ and thus $v \wedge v' = v$, so the function $\kappa : \Var \to \{x, z\}$ defined by $\kappa(x) := x, \kappa(y) := x, \kappa(z) := z$ (and arbitrarily otherwise) is such that $\T_{\V\Cat}$ entails $x \sim_{v \wedge v'} z \Longrightarrow x \sim_v x$ and $x \sim_{v \wedge v'} z \Longrightarrow x \sim_{v'} z$.

Consider also the relational Horn theory $\T_{\PMet_\V}$ over $\Pi_\V$ of Example \ref{without_equality_examples}.\ref{quantale}, which is a schematic extension of $\T_{\Pi_\V} = \T_{\V\RGph}$ given by the set consisting of the generalized transitivity axiom schema and the symmetry axiom schema of Example \ref{schema_egg}. Again, the former axiom schema is not safe in general (because otherwise $\T_{\PMet_\V}\Mod \cong \PMet_\V$ would be cartesian closed by Theorem \ref{lcc_thm_V} below, which is not true in general). But the symmetry axiom schema is (evidently) very safe.  
}
\end{egg}

\begin{prop}
\label{schema_safe_prop}
Let $\T$ be a schematic extension of $\T_\Pi$ given by a set $\scrS$ of axiom schemas over $\Pi$, and let $f : X \to Z$ be a morphism of $\T$-models. Then $f$ is convex with respect to all very safe axiom schemas in $\scrS$.
\end{prop}

\begin{proof}
Let $(\Phi, \calS v_1 \ldots v_n, \sigma)$ be a very safe axiom schema in $\scrS$, and let $\Rbar \in \Pi(n)^\Phi$. Let $\kappa_Z : \Var \to |Z|$ be a valuation satisfying $Z \models \kappa_Z \cdot \varphi_{R_\varphi}$ for each $\varphi \in \Phi$, let $x_i \in f^{-1}(\kappa_Z(v_i))$ for each $1 \leq i \leq n$, let $T \in \Pi(n)$ be such that $T \leq \sigma\left(\Rbar\right)$ and $X \models Tx_1\ldots x_n$, and let us show that $T \leq \bigvee\{ R_\kappa \mid \kappa : \Var \to |X| \text{ is good}\}$. The valuation $\kappa : \Var \to |X|$ given by $\kappa(v_i) := x_i$ for each $1 \leq i \leq n$ is good (since $\Var(\Phi) \subseteq \{v_1, \ldots, v_n\}$), so it suffices to show that $T \leq R_\kappa$, i.e.~that
\[ T \leq \bigvee\left\{\sigma\left(\Rbar \wedge \Sbar\right) \mid \Sbar \in \Pi(n)^\Phi \text{ and } X \models \kappa \cdot \varphi_{S_\varphi} \ \forall \varphi \in \Phi\right\}. \] From $X \models Tx_1 \ldots x_n$ we obtain $X \models \left(\sigma\left(\Rbar\right) \wedge T\right)x_1\ldots x_n$ (since $X$ is a model of $\T_\Pi$) and then $X \models \sigma\left(\Rbar \wedge T\right)x_1\ldots x_n$, since $\sigma\left(\Rbar\right) \wedge T = \sigma\left(\Rbar \wedge T\right)$. Since $(\Phi, \calS v_1 \ldots v_n, \sigma)$ is very safe, it follows that $\T$ entails $\sigma\left(\Rbar \wedge T\right)v_1 \ldots v_n \Longrightarrow \varphi_{R_\varphi \wedge T}$ for each $\varphi \in \Phi$. Since $X$ is a $\T$-model, we then deduce that $X \models \kappa \cdot \varphi_{R_\varphi \wedge T}$ for each $\varphi \in \Phi$. We then have $T = \sigma\left(\Rbar\right) \wedge T = \sigma\left(\Rbar \wedge T\right) = \sigma\left(\Rbar \wedge \left(\Rbar \wedge T\right)\right) \leq \bigvee\left\{\sigma\left(\Rbar \wedge \Sbar\right) \mid \Sbar \in \Pi(n)^\Phi \text{ and } X \models \kappa \cdot \varphi_{S_\varphi} \ \forall \varphi \in \Phi\right\}$, as desired.     
\end{proof}

\noindent Definition \ref{schema_convex} now specializes to \emph{$\T_\Pi$-models} as follows:

\begin{defn}
\label{schema_convex_object}
{\em
Let $X$ be a $\T_\Pi$-model, let $(\Phi, \calS v_1 \ldots v_n, \sigma)$ be an axiom schema over $\Pi$, and let $\left(\Rbar \in \Pi(n)^\Phi, \Phi_{\Rbar}\Longrightarrow \sigma\left(\Rbar\right)v_1\ldots v_n\right)$ be an instance. We say that $X$ is \textbf{convex with respect to the instance $\left(\Rbar, \Phi_{\Rbar}\Longrightarrow \sigma\left(\Rbar\right)v_1\ldots v_n\right)$ of the axiom schema $(\Phi, \calS v_1 \ldots v_n, \sigma)$} if $X$ satisfies the following condition:

Let $x_1, \ldots, x_n \in |X|$, and let us say that a valuation $\kappa : \Var \to |X|$ is \textbf{good} if $\kappa(v_i) = x_i$ for each $1 \leq i \leq n$. For each good valuation $\kappa : \Var \to |X|$, we define
\[ R_\kappa := \bigvee\left\{\sigma\left(\Rbar \wedge \Sbar\right) \mid \Sbar \in \Pi(n)^\Phi \text{ and } X \models \kappa \cdot \varphi_{S_\varphi} \ \forall \varphi \in \Phi\right\}. \] Then for each $T \in \Pi(n)$ such that $T \leq \sigma\left(\Rbar\right)$ and $X \models Tx_1 \ldots x_n$, we require that
\[ T \leq \bigvee\{R_\kappa \mid \kappa : \Var \to |X| \text{ is good}\}. \] We say that $X$ is \textbf{convex with respect to an axiom schema} if $X$ is convex with respect to each instance of the axiom schema.
}  
\end{defn}

\noindent Theorem \ref{schema_convex_exp} now specializes to yield the following:

\begin{theo}
\label{schema_convex_exp_object}
Let $\T$ be a schematic extension of $\T_\Pi$. Then every convex $\T$-model is an exponentiable object of $\T\Mod$. \qed
\end{theo}

\begin{prop}
\label{schema_safe_object_prop}
Let $\T$ be a schematic extension of $\T_\Pi$ given by a set $\scrS$ of axiom schemas over $\Pi$, and let $X$ be a $\T$-model. Then $X$ is convex with respect to all safe axiom schemas in $\scrS$.  
\end{prop}

\begin{proof}
Let $(\Phi, \calS v_1\ldots v_n, \sigma)$ be a safe axiom schema in $\scrS$. Then there is some function $\lambda : \Var \to \{v_1, \ldots, v_n\}$ that fixes $\{v_1, \ldots, v_n\}$ such that $\T$ entails $\sigma\left(\Rbar\right)v_1\ldots v_n \Longrightarrow \lambda \cdot \varphi_{R_\varphi}$ for all $\Rbar \in \Pi(n)^\phi$ and $\varphi \in \Phi$. To show that $X$ is convex with respect to an instance $\left(\Rbar, \Phi_{\Rbar} \Longrightarrow \sigma\left(\Rbar\right)v_1\ldots v_n\right)$, let $x_1, \ldots, x_n \in |X|$, let $T \leq \sigma\left(\Rbar\right)$ be such that $X \models Tx_1 \ldots x_n$, and let us show that $T \leq \bigvee\{ R_\kappa \mid \kappa : \Var \to |X| \text{ is good}\}$. Let $\iota : \{v_1, \ldots, v_n\} \to |X|$ be the function defined by $\iota(v_i) := x_i$ for each $1 \leq i \leq n$. We then have a good valuation $\kappa := \iota \circ \lambda : \Var \to |X|$, so it suffices to show that $T \leq R_{\kappa}$, i.e.~that
$T \leq \bigvee\left\{\sigma\left(\Rbar \wedge \Sbar\right) \mid \Sbar \in \Pi(n)^\Phi \text{ and } X \models \kappa \cdot \varphi_{S_\varphi} \ \forall \varphi \in \Phi\right\}$. From $X \models Tx_1\ldots x_n$ we obtain $X \models \left(\sigma\left(\Rbar\right) \wedge T\right)x_1 \ldots x_n$ (because $X$ is a model of $\T_\Pi$) and then $X \models \sigma\left(\Rbar \wedge T\right)x_1 \ldots x_n$  (since $\sigma\left(\Rbar\right) \wedge T = \sigma\left(\Rbar \wedge T\right)$). We then deduce (from the hypothesis on $\lambda$) that $X \models \kappa \cdot \varphi_{R_\varphi \wedge T}$ for each $\varphi \in \Phi$, so that $T = \sigma\left(\Rbar\right) \wedge T = \sigma\left(\Rbar \wedge T\right) = \sigma\left(\Rbar \wedge \left(\Rbar \wedge T\right)\right) \leq \bigvee\left\{\sigma\left(\Rbar \wedge \Sbar\right) \mid \Sbar \in \Pi(n)^\Phi \text{ and } X \models \kappa \cdot \varphi_{S_\varphi} \ \forall \varphi \in \Phi\right\}$, as desired.     
\end{proof}

\noindent From Theorem \ref{schema_convex_exp}, Proposition \ref{schema_safe_prop}, Theorem  \ref{schema_convex_exp_object}, and Proposition \ref{schema_safe_object_prop} we immediately deduce the following result:

\begin{theo}
\label{lcc_thm_V}
Let $\T$ be a schematic extension of $\T_\Pi$ given by a set $\scrS$ of axiom schemas over $\Pi$. If each axiom schema in $\scrS$ is safe, then $\T\Mod$ is cartesian closed. If each axiom schema in $\scrS$ is very safe, then $\T\Mod$ is moreover locally cartesian closed. \qed
\end{theo}

\noindent We also have the following analogue of Theorem \ref{quasitopos_thm}, whose proof is identical to that of Theorem \ref{quasitopos_thm} after using Theorem \ref{lcc_thm_V} in place of Theorem \ref{lcc_thm}; the result that $\V\RGph$ is a topological universe (and in particular a quasitopos) recovers \cite[Theorem 2.5]{Exponentiation_in_V_categories} (see also \cite[Theorem 5.6]{Exponentiability_lax_algebras}).

\begin{theo}
\label{quasitopos_thm_2}
Let $\T$ be a schematic extension of $\T_\Pi$ given by a set $\scrS$ of axiom schemas over $\Pi$, and suppose that $\T$ is without equality and that each axiom schema in $\scrS$ is very safe. Then $\T\Mod$ is a quasitopos, and moreover a topological universe. In particular, if $(\V, \leq, \tensor, \top)$ is a commutative unital quantale such that $(\V, \leq)$ is a complete Heyting algebra, then $\T_{\Pi_\V}\Mod = \T_{\V\RGph}\Mod \cong \V\RGph$ is a topological universe. \qed 
\end{theo}

\begin{rmk}[\textbf{Further directions}]
\label{further}
{\em
We conclude the paper by discussing two further questions that could be pursued. As we explained in Remarks \ref{nec_rmk} and \ref{nec_rmk_2}, our results only establish the \emph{sufficiency} of convexity for the exponentiability of objects and morphisms in various categories of relational structures, but for certain (classes of) examples it has been established (in \cite{Exponentiation_in_V_categories} and \cite{Exp_morphisms}) that convexity is also \emph{necessary} for exponentiability. It would therefore be interesting and useful to try to characterize the relational Horn theories $\T$ such that convexity is not only sufficient but also necessary for exponentiability in $\T\Mod$.  

Another further question is the following. We have studied convexity when $\T$ is a (certain kind of) relational Horn theory over a preordered relational signature $\Pi$ that is \emph{discrete} (in \S\ref{disc_section}) or a \emph{complete Heyting algebra} (in \S\ref{non_disc_section}). It would be interesting and useful to try to develop an approach to convexity that generalizes both cases and can be applied when $\Pi$ is an arbitrary preordered relational signature. The two definitions of convexity in Definitions \ref{convex_unified} and \ref{schema_convex} are rather different; in particular, the definition of convexity in Definition \ref{schema_convex}  (when $\Pi$ is a complete Heyting algebra) makes explicit and significant use of the complete lattice structure of $\Pi$, which is not available when $\Pi$ is discrete.    
}
\end{rmk}

\bibliographystyle{amsplain}
\bibliography{mybib}

\end{document}